\documentclass[12pt]{amsart}
\usepackage{amsmath,amsthm,amssymb,amsfonts,amscd}
\usepackage{mathrsfs}
\usepackage{bbm}
\usepackage{bbding}
\usepackage{graphicx,latexsym}
\usepackage[backref=page]{hyperref}
\usepackage{geometry}
\usepackage{color}
\usepackage{xcolor}

\numberwithin{equation}{section}

\theoremstyle{plain}
\newtheorem{theorem}{Theorem}[section]
\newtheorem{lemma}[theorem]{Lemma}
\newtheorem{corollary}[theorem]{Corollary}

\theoremstyle{definition}

\theoremstyle{remark}
\newtheorem{remark}[theorem]{Remark}

\renewcommand{\Re}{\operatorname{Re}}
\renewcommand{\Im}{\operatorname{Im}}

\newcommand{\sgn}{\operatorname{sgn}}

\newcommand{\supp}{\operatorname{supp}}

\newcommand{\GL}{\operatorname{GL}}
\newcommand{\SL}{\operatorname{SL}}

\newcommand{\dd}{\mathrm{d}}

\newcommand{\Sym}{\operatorname{Sym}}

\makeatletter
\def\@tocline#1#2#3#4#5#6#7{\relax
  \ifnum #1>\c@tocdepth 
  \else
    \par \addpenalty\@secpenalty\addvspace{#2}%
    \begingroup \hyphenpenalty\@M
    \@ifempty{#4}{%
      \@tempdima\csname r@tocindent\number#1\endcsname\relax
    }{%
      \@tempdima#4\relax
    }%
    \parindent\z@ \leftskip#3\relax \advance\leftskip\@tempdima\relax
    \rightskip\@pnumwidth plus4em \parfillskip-\@pnumwidth
    #5\leavevmode\hskip-\@tempdima
      \ifcase #1
       \or\or \hskip 1em \or \hskip 2em \else \hskip 3em \fi%
      #6\nobreak\relax
    \hfill\hbox to\@pnumwidth{\@tocpagenum{#7}}\par
    \nobreak
    \endgroup
  \fi}
\makeatother

\makeatletter
\@namedef{subjclassname@2020}{%
  \textup{2020} Mathematics Subject Classification}
\makeatother

\begin{document}

\title[The cubic moment of Hecke--Maass cusp forms and moments of $L$-functions]
{The cubic moment of Hecke--Maass cusp forms and \\ moments of $L$-functions}
\author{Bingrong Huang}
\address{Data Science Institute and School of Mathematics \\ Shandong University \\ Jinan \\Shandong 250100 \\China}
\email{brhuang@sdu.edu.cn}
\date{\today}
\thanks{This work was supported by  the National Key R\&D Program of China (No. 2021YFA1000700) and NSFC (Nos. 12001314 and 12031008).}

\begin{abstract}
  In this paper, we prove that the smooth cubic moments dissipate for the Hecke--Maass cusp forms, which gives a new case of the random wave conjecture.
  In fact, we can prove a polynomial decay for the smooth cubic moments,
  while for the smooth second moment (i.e.\ QUE) no rate of decay is known unconditionally for general Hecke--Maass cusp forms.   The proof is based on various estimates of moments of central $L$-values. We prove the Lindel\"of on average bound for the first moment of $\GL(3)\times \GL(2)$ $L$-functions in short intervals of the subconvexity strength length, and the convexity strength upper bound for the mixed moment of $\GL(2)$ and the triple product $L$-functions.
  In particular, we prove new subconvexity bounds of certain $\GL(3)\times \GL(2)$ $L$-functions.
\end{abstract}


\maketitle

\section{Introduction} \label{sec:Intr}


\subsection{Cubic moment of Hecke--Maass cusp forms}

The value distribution of   eigenfunctions of the Laplacian on a Riemann surface has received a lot of attention in the context
of quantum chaos. Following Michael Berry's suggestion \cite{berry1977rugular} that eigenfunctions for chaotic systems are
modeled by random waves, it is believed that eigenfunctions on a compact hyperbolic surface have a Gaussian value distribution as the eigenvalue tends to infinity, and the  moments of an $L^2$-normalized eigenfunction should be given by the Gaussian moments.

Here we examine the modular surface $\mathbb{X}=\SL(2,\mathbb Z)\backslash \mathbb H$,  with $\mathbb H=\{z=x+iy:y>0\}$, which  is a noncompact, but finite area hyperbolic surface. The spectrum of the Laplacian
$\Delta = -y^2(\frac{\partial^2}{\partial x^2}+\frac{\partial^2}{\partial y^2})$
on $\mathbb X$ has both discrete and continuous components. The discrete spectrum consists of the
constants and the space  of cusp forms, for which we can take an orthonormal
basis $\{\phi_j\}_{j=1}^\infty$  of Hecke--Maass forms, which are real valued joint eigenfunctions of both the Laplacian and all the Hecke operators, and decay exponentially at the cusp of the modular domain. The modular surface $\mathbb X$ carries a further symmetry induced by the orientation reversing isometry $z\rightarrow -\bar{z}$ of $\mathbb{H}$ and our $\phi_j$'s are either even or odd with respect to this symmetry.

The non-compact nature of the modular domain prevents a literal application of the random wave model; for instance Sarnak \cite{sarnak2004letter} proved
that the supremum of cusp forms is much larger than that of random waves: $\sup_{z\in \mathbb X}|\phi(z)| \gg \lambda_\phi^{1/12}$ for $\phi\in \{\phi_j\}$, due to blowup at the cusp.
Here $\lambda_\phi$ is the Laplacian eigenvalue of $\phi$, so that $\Delta\phi=\lambda_\phi \phi$.
However, we do expect that the random wave model remains valid when we restrict to a compact subset, which we do now. We wish to examine the value distribution of $\phi(z)$ when $z$ is sampled from a compactly supported probability distribution, having a smooth density $\psi(z)\in C_c^\infty(\mathbb X)$ with respect to the  hyperbolic measure $ \frac{\dd x \dd y}{y^2}$, where $\psi\geq 0$ and $\int_{\mathbb X} \psi(z)  \frac{\dd x \dd y}{y^2} =1$.
In particular we study the cubic moments, which detect sign changes, and should vanish if the value distribution is Gaussian. Our main result confirms this, with a power saving.

\begin{theorem}\label{thm:HM}
Fix a smooth compactly supported function  $\psi$  on $\mathbb X$. Then as $\lambda_\phi \to \infty$, for any $\delta<1/24$,
\[
  \int_{\mathbb X}  \psi(z)  \phi(z)^3  \frac{\dd x \dd y}{y^2} \ll \lambda_\phi^{-\delta}.
\]
\end{theorem}

The quantum unique ergodicity (QUE) \cite{RudnickSarnak1994behaviour} concerns the second moment of eigenfunctions.
It asserts that  we have
\begin{equation}\label{eqn:QUE}
  \int_{\mathbb{X}} \psi(z) \phi(z)^2 \frac{\dd x\dd y}{y^2} = \int_{\mathbb{X}} \psi(z) \frac{3}{\pi}\frac{\dd x\dd y}{y^2} + o(1),  \quad  \textrm{as }  \lambda_\phi\rightarrow \infty,
\end{equation}
for any smooth compactly supported function $\psi$ on $\mathbb X$.
For Hecke--Maass cusp forms for $\SL(2,\mathbb{Z})$, this is known to be true due to the breakthroughs  of Lindenstrauss \cite{lindenstrauss2006invariant} and Soundararajan \cite{soundararajan2010quantum}. But no rate of decay of the error term is known.

The cubic moment is also related to the $L^4$-norm problem of Hecke--Maass forms.
Under the generalized Lindel\"of Hypothesis (GLH),
Sarnak and Watson \cite[Theorem 3]{sarnak2003spectra} can prove a sharp upper bound for the $L^4$-norm, i.e., $\|\phi\|_4^4=\int_{\mathbb{X}} \phi(z)^4 \frac{\dd x\dd y}{y^2}\ll \lambda_\phi^{o(1)}$.
Buttcane--Khan \cite{ButtcaneKhan2017fourth} proved the $L^4$-norm of $\phi$ on the whole surface $\mathbb{X}$ has an asymptotic formula $\|\phi\|_4^4 \sim 9/\pi$, under the GLH.
Recently, Humphries--Khan \cite{HumphriesKhan} proved an upper bound unconditionally, i.e.,
$\|\phi\|_4 \ll \lambda_\phi^{3/304+o(1)}$.
The random wave model predicts
$
  \int_{\mathbb X} \psi(z) \phi(z)^4 \frac{\dd x\dd y}{y^2} \sim 3 \int_{\mathbb{X}} \psi(z) \frac{9}{\pi^2}\frac{\dd x\dd y}{y^2},
$
for any smooth compactly supported function  $\psi$  on $\mathbb X$.

%


\begin{remark}
It is a surprise that a power saving upper bound can be proved for the smooth cubic moment $\int_{\mathbb X}  \psi(z)  \phi(z)^3  \frac{\dd x \dd y}{y^2}$,
since for the second moment (QUE) one cannot yet obtain any rate of decay unconditionally.
Assume the generalized Riemann Hypothesis or GLH, the exponent in Theorem \ref{thm:HM} can be taken to be  any $\delta<1/4$.
But unlike the QUE case, we may not expect $1/4$ is sharp.
\end{remark}

\begin{remark}
  Watson \cite[Theorem 5]{watson2008rankin} proved $\int_{\mathbb X} \phi(z)^3 \frac{\dd x\dd y}{y^2} 
  \ll \lambda_\phi^{-1/12+\varepsilon}$. 
  This is a simple consequence of the subconvexity bound of the $\GL(2)$ $L$-function $L(1/2,\phi)$ and Watson's formula. But to estimate the smooth cubic moment is much harder, since we have more complicated terms involved the Hecke--Maass cusp forms in \eqref{eqn:Selberg_decomposition}, and Watson's result corresponds to the contribution from the constant function.
\end{remark}



\begin{remark}
  Note that $\lim_{\lambda_\phi\rightarrow\infty} \int_{\mathbb X} \psi(z) \phi(z)^3 \frac{\dd x\dd y}{y^2}$ vanishes because of the cancellation between positive and negative values of $\phi$.
  Fix a $\psi(z)\in C_c^\infty(\mathbb X)$ with   $\psi\geq 0$ and $\int_{\mathbb X} \psi(z)  \frac{\dd x \dd y}{y^2} =1$.
  By H\"older's inequality, we have
\[
  \int_{\mathbb X} \psi(z)\phi(z)^2 \frac{\dd x\dd y}{y^2} \leq \Big( \int_{\mathbb X} \psi(z) |\phi(z)|^3 \frac{\dd x\dd y}{y^2} \Big)^{2/3}
  \Big( \int_{\mathbb X} \psi(z) \frac{\dd x\dd y}{y^2} \Big)^{1/3}.
\]
By the QUE \eqref{eqn:QUE}, we get
\begin{equation}\label{eqn:L3norm_HM}
  \int_{\mathbb X} \psi(z) |\phi(z)|^3 \frac{\dd x\dd y}{y^2}
  \geq  
   \left(\frac{3}{\pi}\right)^{3/2}  + o(1),
\end{equation}
which is bounded below by a constant.
\end{remark}

The Eisenstein series $E(z,s)$ is defined by
\[
  E(z,s) = \sum_{\gamma\in \Gamma_\infty\backslash \SL(2,\mathbb{Z})} \Im(\gamma z)^s = \frac{1}{2} \sum_{(c,d)=1} \frac{y^s}{|cz+d|^{2s}}
\]
for $\Re(s)>1$, and has a meromorphic continuation to $s\in\mathbb{C}$, where
$\Gamma_\infty=\left\{\pm \left(\begin{smallmatrix}1&n\\ &1\end{smallmatrix}\right):n\in\mathbb{Z}\right\}$.
The continuous spectrum is spanned by Eisenstein series $E(z,1/2+it)$ with $t\in \mathbb{R}$.
To prove Theorem \ref{thm:HM} we apply Selberg's spectral decomposition of $\psi$,
\begin{equation}\label{eqn:Selberg_decomposition}
  \psi(z) = \langle \psi,1 \rangle \frac{3}{\pi} + \sum_{k\geq1} \langle \psi, \phi_k\rangle \phi_k(z)
  + \frac{1}{4\pi} \int_{\mathbb{R}} \langle \psi,E(\cdot,1/2+it)\rangle E(z,1/2+it) \dd t,
\end{equation}
where $\langle \psi,f\rangle = \int_{\mathbb{X}} \psi(z)\overline{f(z)} \frac{\dd x\dd y}{y^2}$ is the Petersson inner product on $L^2(\mathbb{X})$.
The contribution from the constant function to $\langle \psi,\phi^3 \rangle$ are estimated by Watson \cite[Theorem 5]{watson2008rankin}. We will consider the contribution from both the cusp forms and Eisenstein series and prove the following theorem.

\begin{theorem}\label{thm:HM2}
  Assume $\lambda_k\ll \lambda_\phi^{o(1)}$ and real $t\ll \lambda_\phi^{o(1)}$. Then we have
  \begin{equation}\label{eqn:3moment_cusp}
    \int_{\mathbb X} \phi_k(z) \phi(z)^3 \frac{\dd x \dd y}{y^2} \ll \lambda_\phi^{-1/24+o(1)},
  \end{equation}
  and
  \begin{equation}\label{eqn:3moment_ES}
    \int_{\mathbb X} E(z,1/2+it) \phi(z)^3 \frac{\dd x\dd y}{y^2} \ll \lambda_\phi^{-1/12+o(1)},
  \end{equation}
  as $\lambda_\phi\rightarrow \infty$.
\end{theorem}

To prove Theorems \ref{thm:HM} and \ref{thm:HM2}, we will use Watson's formula and the Rankin--Selberg theory to reduce them to moments of $L$-functions, which will be discussed in \S\ref{subsec:Moments_L_values}.
To prove \eqref{eqn:3moment_cusp}, by Parseval's formula we get
\begin{align*}
  \int_{\mathbb X} \phi_k(z) \phi(z)^3 \frac{\dd x \dd y}{y^2}
  = \langle \phi^2 , \frac{3}{\pi} \rangle \langle  1, \phi_k \phi\rangle
  & +
  \sum_{j\geq1} \langle \phi^2 , \phi_j \rangle \langle  \phi_j, \phi_k \phi\rangle
  \\
  & + \frac{1}{4\pi} \int_{\mathbb{R}}
  \langle \phi^2 , E(\cdot,1/2+i\tau) \rangle \langle  E(\cdot,1/2+i\tau), \phi_k \phi\rangle \dd \tau.
\end{align*}
For the cusp form contribution, by taking absolute value and applying Watson's formula, we arrive at
\begin{equation}\label{eqn:moment_half}
  \sum_{j\geq1} h(t_j,t_\phi,t_k) L(1/2,\phi_j)^{1/2} L(1/2,\phi_j \times \Sym^2 \phi)^{1/2} L(1/2,\phi_j \times \phi\times \phi_k)^{1/2},
\end{equation}
for a certain weight function $h(t_j,t_\phi,t_k)$ with exponential decay if $t_j-t_\phi \gg t_\phi^\varepsilon $ and of size $t_\phi^{-3/2}$ if $t_j-t_\phi\ll t_\phi^\varepsilon $.
Here it is crucial to use the nonnegativity of these central $L$-values.
See the beginning of \S \ref{sec:first_moment_L_functions} and \S \ref{sec:mixed_moment_L_functions} for the definitions of these $L$-functions.
To estimate \eqref{eqn:moment_half}, we will apply the Cauchy--Schwarz inequality to reduce it to the mixed moment of $L(1/2,\phi_j) L(1/2,\phi_j \times \phi\times \phi_k)$ and the first moment of $L(1/2,\phi_j \times \Sym^2 \phi)$.
Essentially, the cubic moment problem is reduced to the subconvexity problem of $L(1/2,\phi_j \times \Sym^2 \phi)$ with $t_j-t_\phi\ll  t_\phi^\varepsilon $.
Note that in our problem we have a family of $L$-functions of size $t_\phi$; we are lucky that the moment method works here. But for the quantitative QUE problem of Hecke--Maass cusp forms, we face the subconvexity problem of $L(1/2,\phi_j \times \Sym^2 \phi)$ with $t_j \ll  t_\phi^\varepsilon $, which is still open.

\medskip

We can do slightly better for the cubic moments of Eisenstein series and dihedral Maass forms. 
QUE for Eisenstein series was proved by Luo--Sarnak \cite{LuoSarnak1995quantum}: 
for any fixed smooth and compactly supported function $\psi:\mathbb{X}\rightarrow\mathbb{C}$, we have
\begin{equation}\label{eqn:QUE_ES}
   \langle \psi,|E(\cdot,1/2+iT)|^2 \rangle = \int_{\mathbb{X}} \psi(z) |E(z,1/2+iT)|^2  \frac{\dd x\dd y}{y^2}
   \sim \frac{6}{\pi} \langle\psi,1\rangle \log T  ,
\end{equation}
as $T\rightarrow\infty$.
Sharp $L^4$-norm bounds were proved by Spinu \cite{spinu2003norm}, Humphries \cite{humphries2017equidistribution}, and Djankovi\'c--Khan \cite{DK}.
We have
\begin{equation}\label{eqn:L4_ES}
  \int_{\mathbb X} \psi(z) |E(z,1/2+iT)|^4 \frac{\dd x\dd y}{y^2} \ll (\log T)^2,
\end{equation}
for any fixed smooth compactly supported function $\psi$ on $\mathbb X$.
  Let $T\geq10$. Then by using Zagier's regularized inner product \cite{zagier1981rankin} and our method in this paper, we can at least prove
  \begin{equation}\label{eqn:cubic_moment_ES}
    \int_{\mathbb X} \psi(z) E(z,1/2+iT)^3 \frac{\dd x\dd y}{y^2} \ll T^{-\delta},
  \end{equation}
  for any $\delta<1/6$.
%
Assume $\psi\geq0$ and $\int_{\mathbb X} \psi(z)  \frac{\dd x \dd y}{y^2} =1$.
As in \eqref{eqn:L3norm_HM},
by  \eqref{eqn:QUE_ES} and  \eqref{eqn:L4_ES}, we have
\begin{equation}\label{eqn:L3norm_ES}
  (\log T)^{3/2} \ll \int_{\mathbb X} \psi(z)  |E(z,1/2+iT)|^3 \frac{\dd x\dd y}{y^2} \ll (\log T)^{3/2}.
\end{equation}
Since we have the optimal upper bound of $L^4$-norm of Eisenstein series \eqref{eqn:L4_ES}, by smooth approximations of an indicator function and \eqref{eqn:cubic_moment_ES}, for any fixed compact domain $\Omega\subset \mathbb{X}$ with measure zero boundary, we have
\[
  \int_{\Omega}  E(z,1/2+iT)^3 \frac{\dd x\dd y}{y^2}  = o((\log T)^{3/2}),
  \quad  \textrm{as } T\rightarrow\infty.
\]

For  dihedral forms, QUE was proved by Sarnak \cite{sarnak2001estimates} and Liu--Ye \cite{LY}, with a polynomial decay of the error term.
Stronger $L^4$-norm estimates for dihedral Maass forms were proved by Luo \cite{luo2014norms}. Recently,  Humphries--Khan \cite{HumphriesKhan2020} proved an asymptotic formula for the $L^4$-norm of dihedral Maass forms.
We should prove a better  exponent for dihedral Maass forms than Theorem \ref{thm:HM}. \emph{Cf.}  the Eisenstein series case \eqref{eqn:cubic_moment_ES}.
The reason is that the symmetric square $L$-function of a dihedral form can be decomposed, so that we can reduce the cubic moment  to other simpler moments of $L$-functions.



\subsection{Moments of $L$-functions}\label{subsec:Moments_L_values}

Estimating moments of $L$-functions is one of the central problems in number theory.
Let $\phi$ be a Hecke--Maass cusp form with spectral parameter $t_\phi$.
Let $\{\phi_j\}_{j\geq1}$ be an orthonormal basis of Hecke--Maass forms for $\SL(2,\mathbb{Z})$.
The first moment of $\GL(3)\times \GL(2)$ $L$-functions in short intervals
\[
    \sum_{T-M\leq t_j \leq T+M}  L(1/2,\phi_j \times \Sym^2 \phi)
    + \int_{T-M}^{T+M} |L(1/2+it , \Sym^2 \phi)|^2 \dd t
\]
was considered in Li \cite{li2011bounds} for a fixed self dual $\GL(3)$ form $\Sym^2\phi$. Note that  $t_\phi\ll 1$. Li can prove Lindel\"of on average bounds as $T$ goes to infinity when  $ T^{3/8+\varepsilon}\leq M \leq T^{1/2-\varepsilon}$, from which she obtained subconvexity bounds for those $L$-functions. McKee--Sun--Ye \cite{MSY} improved Li's result to allow $M=T^{1/3+\varepsilon}$. The best known result in this setting is given by Lin--Nunes--Qi \cite{LNQ} very recently, which allows $M=T^{1/5+\varepsilon}$.

In this paper, to prove Theorem \ref{thm:HM2} we should consider the case  $t_\phi=T$.
The coefficients of $\Sym^2 \phi$ in this setting become more complicated, as the conductors will be larger. So to prove Lindel\"of on average bounds for short moments in subconvexity strength range is challenging.
We will extend the ideas in the above mentioned works and our techniques in \cite{huang2021rankin} and \cite{huang2021uniform} to prove Lindel\"of on average bounds when $T^{1/3+\varepsilon} \leq M \leq T^{1/2-\varepsilon}$.
Our technical main results are the following two estimates of moments of central $L$-values, which will be used to prove \eqref{eqn:3moment_cusp}.

\begin{theorem}\label{thm:moment_L_functions_6}
  Let $\phi$ be a Hecke--Maass cusp form with the spectral parameter $T>0$.
  Let $T^{1/3+\varepsilon}\leq M\leq T^{1/2-\varepsilon}$. Then we have
  \[
    \sum_{T-M\leq t_j \leq T+M}  L(1/2,\phi_j \times \Sym^2 \phi)
    + \int_{T-M}^{T+M} |L(1/2+it , \Sym^2 \phi)|^2 \dd t \ll T^{1+\varepsilon}M.
  \]
  In particular, we have
  \[
    \sum_{T-T^\varepsilon\leq t_j \leq T+T^\varepsilon}  L(1/2,\phi_j \times \Sym^2 \phi) \ll T^{4/3+\varepsilon}.
  \]
\end{theorem}

\begin{theorem}\label{thm:moment_L_functions_2+8}
  Let $\phi$ be a Hecke--Maass cusp form with the spectral parameter $T>0$.
  Let $\phi_k$ be a Hecke--Maass cusp form with the spectral parameter $t_k>0$.
  Assume $t_k\leq T^\varepsilon$.  Then we have
  \[
    \sum_{T-T^{\varepsilon}\leq t_j \leq T+T^{\varepsilon}} L(1/2,\phi_j) L(1/2,\phi_j \times \phi\times \phi_k) \ll T^{3/2+\varepsilon}.
  \]
\end{theorem}

Our first novelty here is that in \eqref{eqn:moment_half} we group $L(1/2,\phi_j)$ and  $L(1/2,\phi_j \times \phi\times \phi_k)$ together and find the convexity strength upper bound in Theorem \ref{thm:moment_L_functions_2+8}. We can prove Theorem \ref{thm:moment_L_functions_2+8} because the analytic conductor of $L(1/2,\phi_j \times \phi\times \phi_k)$ is relatively small thanks to the conductor dropping phenomenon $t_j=T+O(T^\varepsilon)$ and that the factor $L(1/2,\phi_j)$ is relatively simple. The proof is based on the Cauchy--Schwarz inequality and the  spectral large sieve.
Since we do not have a good Dirichlet series expression of $L(s,\phi_j \times \phi\times \phi_k)$ in terms of the Fourier coefficients of $\phi_j,\phi,\phi_k$, we will use the fact $\Phi=\phi\times \phi_k$ is a cusp form for $\SL(4,\mathbb{Z})$ due to Ramakrishnan \cite{Ramakrishnan}.

Our second novelty here is that we enlarge the interval of the moment in Theorem \ref{thm:moment_L_functions_6}. To prove \eqref{eqn:3moment_cusp} we actually need the case with $M=T^\varepsilon$ as in Theorem \ref{thm:moment_L_functions_2+8}.
However, it seems hard to prove nontrivial upper bound when $M=T^\varepsilon$ directly.
Because of the nonnegativity of the central values $L(1/2,\phi_j \times \Sym^2 \phi)$ (see Lapid \cite{lapid2003}), we are allowed to consider longer interval averages.
\emph{Cf.} Lin--Nunes--Qi \cite{LNQ}.
Noting that as long as $M\ll T^{1-\varepsilon}$, the size of the analytic conductor of $L(1/2,\phi_j \times \Sym^2 \phi)$ will not change.  If $M=T^{1/2+\varepsilon}$, the off-diagonal terms with $J$-Bessel function will be arbitrarily small. Hence one may prove the Lindel\"of on average bounds for this first moment when $M=T^{1/2+\varepsilon}$ rather easily.
But this is still not enough for our purpose. So we need to deal with the case $M=T^{1/2-\delta}$ for some positive $\delta$.
This moment itself is very interesting, since the Lindel\"of on average bounds can prove the subconvexity bounds for those $L$-functions. We have the following result.

\begin{corollary}\label{cor}
  Let $\phi$ be a Hecke--Maass cusp form with the spectral parameter $T>0$.
  Let $\phi_j$ be a Hecke--Maass cusp form with the spectral parameter $t_j>0$.
  Let $t\in \mathbb{R}$.
  Assume $|t_j-T|\leq T^{1/3}$ and $||t|-T|\leq T^{1/3}$. Then we have
  \[
    L(1/2,\phi_j\times \Sym^2 \phi) \ll T^{4/3+\varepsilon}
  \]
  and
  \[
    L(1/2+it,  \Sym^2 \phi) \ll T^{2/3+\varepsilon}.
  \]
\end{corollary}

\begin{remark}
  The convexity bound is $L(1/2,\phi_j\times \Sym^2 \phi) \ll T^{3/2+\varepsilon}$.
  One may extend our proof of Theorem  \ref{thm:moment_L_functions_6} to more general setting to prove subconvexity bounds of $L(1/2,\phi_j\times \Sym^2 \phi)$ when $\eta T \leq t_j\leq ( 2-\eta)T$ and $t_j\geq (2+\eta)T$ for any fixed small positive constant $\eta$.
  Note that Nelson's work \cite{Nelson} implies a weaker subconvexity bound $L(1/2,\phi_j\times \Sym^2 \phi) \ll T^{3/2-1/13800+\varepsilon}$, but his method works for a much more general setting.
  For the symmetric square $L$-functions, Khan--Young \cite{KY} proved this subconvexity bound by a different method.
\end{remark}


In the end of this subsection, we discuss some difficulties in the proof of Theorem \ref{thm:moment_L_functions_6}. As usual, we will apply the approximate functional equation and Kuznetsov trace formula to the first moment we want to estimate.
Then we need to apply the $\GL(3)$ Voronoi summation formula. We will use a version of Miller--Zhou \cite{MZ} (Lemma \ref{lemma:VSF}) as in Lin--Nunes--Qi's recently work \cite{LNQ}.
The usual method to deal with the integral transforms after the Voronoi summation formula does not work here, since \cite[Lemma 2.1]{li2011bounds} and \cite[Lemma 6]{blomer2012subconvexity} work for fixed $\GL(3)$ forms.

We first consider the terms with $J$-Bessel function.
In \S\ref{subsec:IT}, we deal with this by applying the stationary phase method three times carefully.
We make full use of the observation that the Kuznetsov trace formula provides us a rather explicit and simple oscillating factor (see e.g.\ Young \cite{young2017}), so that we can apply the stationary phase method to the Mellin transform first.
It is important to obtain the factor $e(\mp n/r)$ in the main term of the integral transform which will cancel out with the arithmetic character in the right-hand side of the Voronoi summation formula. This observation goes back to Conrey--Iwaniec \cite{CI}.
Notice that the size of the remaining phase function is $\ll T/M$ instead of $T$. See \eqref{eqn:R+<<sum}.
This is key to our proof. See \cite[\S4]{huang2021uniform} for a similar phenomenon.
Without working out this explicitly, one may imagine that the size should be roughly $T$ which is the size of the maximal spectral parameter of $\Sym^2 \phi$, and then we may not expect to have additional saving by using the $\GL(3)$ Voronoi summation formula again.
After doing this,  typically the $\GL(3)$ $n$-sum will be of size $T^{3/2}$.
Now as in \cite{li2011bounds} and \cite{LNQ}, we need to apply  the $\GL(3)$ Voronoi summation formula for the second time. In fact, we apply the functional equation instead since there is no arithmetic character now. Note that the spectral parameters of $\Sym^2 \phi$ are $0,\pm2iT$. The conductor will be $T^3/M$. So the dual length will be of size $T^{3/2}/M$ and we get additional saving $M^{1/2}$. So the off-diagonals can be bounded by $O(T^{3/2+o(1)}/M^{1/2})$, which will be $O(TM)$ if $M\geq T^{1/3+\varepsilon}$. We remark that the conductor in the final step in our case is larger than the one in \cite{li2011bounds,LNQ} which is roughly $T^3/M^3$.

We still need to estimate the terms with $K$-Bessel function. Unlike the treatments  in \cite{li2011bounds,LNQ}, we have to work hard as in the $J$-Bessel case.
Together with a smooth dyadic decomposition of intervals (see Lemma \ref{lemma:I-}), we still can treat the integrals efficiently. But this time, we cannot deal with the Mellin transform by the stationary phase method. As in \cite{li2011bounds}, we use the Fourier transform to handle the integral coming from the Kuznetsov formula. After a smooth partition of unity, we can estimate one of the integrals coming from the $\GL(3)$ Voronoi summation formula via the stationary phase method. For another one, we can restrict its support. In \cite{li2011bounds}, it is enough to apply  trivial estimates at this step to finish the proof. But in our setting, this is not sufficient. Although there will be no factor $e(\mp n/r)$ to cancel out the additive character $e(\pm n/r)$ as in the $J$-Bessel case (see \eqref{eqn:R-<<sum}), we notice that the size of $n/r$ is the same as the oscillation in the weight function (see \eqref{eqn:R-<<sum_Upsilon}). So we view $e(\pm n/r)$ as an analytic weight function and apply the $\GL(3)$ Voronoi summation formula for the second time as in the $J$-Bessel case. This is comparable to Young \cite{young2017}.

It is possible to give a second proof of Theorem \ref{thm:moment_L_functions_6} by the spectral reciprocity formula. See e.g. \cite{Kwan} and \cite{HumphriesKhan}. As pointed out in \cite[P. 9]{HumphriesKhan}, the short length of the moment means that
the analysis of the size and length of the transform $\mathcal H$ becomes significantly more challenging. Our treatment of the integral transforms may overcome the challenge.

\subsection{Plan for this paper}
The rest of this paper is organized as follows.
In \S \ref{sec:preliminaries}, we give a review of the theory of automorphic forms, Kuznetsov trace formula, and the spectral large sieve inequality.
In \S \ref{sec:first_moment_L_functions}, we prove Theorem \ref{thm:moment_L_functions_6}.
In \S \ref{sec:mixed_moment_L_functions}, we prove Theorem \ref{thm:moment_L_functions_2+8}.
Finally,  in \S \ref{sec:cubic_moment} we will prove Theorems \ref{thm:HM} and \ref{thm:HM2} by using Theorems \ref{thm:moment_L_functions_6} and \ref{thm:moment_L_functions_2+8} and also other moments of $L$-functions.

\medskip
\textbf{Notation.}
Throughout the paper, $\varepsilon$ is an arbitrarily small positive number;
all of them may be different at each occurrence.
As usual, $e(x)=e^{2\pi i x}$.
We use $y\asymp Y$ to mean that $c_1 Y\leq |y|\leq c_2 Y$ for some positive constants $c_1$ and $c_2$.


\section{Preliminaries} \label{sec:preliminaries}

\subsection{Automorphic forms}

Let $\{\phi_j\}_{j\geq1}$ be an orthonormal basis of Hecke--Maass cusp forms for $\SL(2,\mathbb{Z})$  such that $\phi\in\{\phi_j\}_{j\geq1}$.
We can assume all $\phi_j$ are real valued.
Denote the spectral parameter of $\phi_j$ by $t_j>1$ and the Fourier coefficients by $\lambda_j(n)$. Write $s_j=1/2+it_j$.
Let $E(z,s)$ be the standard Eisenstein series. Denote $E_t(z)=E(z,1/2+it)$.
Let $U(z)$ be one of the two functions $\phi_j(z)$ and  $E_t(z)$.
Each $U(z)$ has the Fourier expansion
\begin{equation}
\label{eq:UFourier}
 U(z) = c_0(y) + \rho(1) \sum_{n \neq 0} \frac{\lambda(n)}{\sqrt{|n|}} W_{s}(nz),
\end{equation}
with additional notation as follows.  Here $c_0(y) = c_0(y,s)$ is the constant term in the Fourier expansion which is nonzero only in the Eisenstein case for which
$$c_0(y,s) = y^{s} + \varphi(s) y^{1-s}$$
with
$$
  \varphi(s) = \frac{\theta(1-s)}{\theta(s)}
  \quad \textrm{and} \quad
  \theta(s) = \pi^{-s} \Gamma(s) \zeta(2s) = \Lambda(2s).
$$
Note $|\varphi(1/2 + it)| =1$.
Here $\lambda(n)$ are Hecke eigenvalues which on the Ramanujan conjecture are bounded in absolute value by the divisor function $d(n)$, and
\begin{equation}
 \label{eq:Wdef}
 W_s(z) := 2 \sqrt{|y|} K_{s-1/2}(2\pi|y|)e(x)
\end{equation}
is the Whittaker function.
In the case $U$ is the Eisenstein series $E_t$,
$$\lambda_t(n) = \eta_{t}(n) := \sum_{ab = |n|} (a/b)^{it},$$
and with the above normalization of the constant term, we have
\begin{equation}
 \label{eq:rho1def}
 \rho_t(1) = 1/\theta(1/2 + it).
\end{equation}
Then by Stirling's formula and standard bounds for the Riemann zeta function, we deduce
\begin{equation}
\label{eq:rho1squared}
 |\rho_t(1)|^2 =   \frac{ \cosh(\pi t)}{ |\zeta(1+2it)|^2} = (1+|t|)^{o(1)} \exp(\pi |t|).
\end{equation}
In the case $U$ is a Hecke--Maass cusp form $\phi_j$, we have
\begin{equation}\label{eq:rho1squared-c}
  |\rho_j(1)|^2 =   \frac{ \cosh(\pi t_j)}{L(1,\Sym^2 \phi_j)} = t_j^{o(1)} \exp(\pi t_j).
\end{equation}

\subsection{Kuznetsov trace formula}
Define the harmonic weights
\[
  \omega_j = \frac{4\pi |\rho_j(1)|^2 }{ \cosh(\pi t_j) } = \frac{4\pi}{L(1,\Sym^2 \phi_j)}
  \quad  \textrm{and} \quad
  \omega(t) = \frac{4\pi |\rho_t(1)|^2}{\cosh(\pi t)} = \frac{4\pi}{|\zeta(1+2it)|^2}.
\]
By \cite{iwaniec1990small}, \cite{Hoffstein-Lockhart} and the standard bounds of the Riemann zeta function, we have
\[
  (\log t_j)^{-1} \ll \omega_j \ll \log t_j, \quad
  (\log (1+|t|))^{-1} \ll \zeta(1+2it) \ll \log (1+|t|).
\]
Let
\[
  S(a,b;c) = {\sum_{d(c)}}^* e\left(\frac{ad+b\bar{d}}{c}\right)
\]
be the classical Kloosterman sum. For any test function
$h(t)$ which is even and satisfies the following conditions:
\begin{itemize}
  \item [(i)] $h(t)$ is holomorphic in $|\Im(t)|\leq 1/2+\varepsilon$,
  \item [(ii)] $h(t)\ll (1+|t|)^{-2-\varepsilon}$ in the above strip,
\end{itemize}
we have the following Kuznetsov formula (see~\cite[Eq. (3.17)]{CI} for example).
\begin{lemma}\label{lemma: KTF}
  For $m,n\geq1$, we have
  \begin{equation*}
    \begin{split}
        & {\sum_j}'h(t_j)\omega_j \lambda_j(m)\lambda_j(n) + \frac{1}{4\pi}\int_{-\infty}^{\infty}h(t)\omega(t)\eta_t(m)\eta_t(n)\dd t \\
        & \hskip 120pt = \frac{1}{2}\delta_{m,n}H + \frac{1}{2}\sum_{c\geq1}\frac{1}{c}\sum_{\pm} S(n,\pm m;c)H^{\pm}\left(\frac{4\pi\sqrt{mn}}{c}\right),
    \end{split}
  \end{equation*}
  where $\sum'$ restricts to the even Hecke--Maass cusp forms, $\delta_{m,n}$ is the Kronecker symbol,
  \begin{equation}\label{eqn: H}
    \begin{split}
       H & = \frac{2}{\pi}\int_{0}^{\infty} h(t) \tanh(\pi t)t \dd t, \\
       H^+(x) & = 2i \int_{-\infty}^{\infty} J_{2it}(x)\frac{h(t)t}{\cosh(\pi t)} \dd t, \\
       H^-(x) & = \frac{4}{\pi} \int_{-\infty}^{\infty} K_{2it}(x)\sinh(\pi t)h(t)t \dd t,
    \end{split}
  \end{equation}
  and $J_\nu(x)$ and $K_\nu(x)$ are the standard $J$-Bessel function and $K$-Bessel function respectively.
\end{lemma}

\begin{lemma}\label{lemma: H+}
  Let $H^+$ be given by \eqref{eqn: H}.
  There exists a function $g$ depending on $T$ and $M$ satisfying $g^{(j)}(y)\ll_{j,A} (1+|y|)^{-A}$,
  so that
  \begin{equation}\label{eqn: H^+=int}
    H^+(x) = MT\int_{|v|\leq \frac{M^\varepsilon}{M}}\cos(x\cosh(v))e\left(\frac{vT}{\pi}\right)g(Mv)
    \dd v + O(T^{-A}).
  \end{equation}
  Furthermore, $H^+(x)\ll T^{-A}$ unless $x\gg MT^{1-\varepsilon}$.
  In addition, for $x\ll T$, we have $H^+(x)\ll Mx$.
\end{lemma}

\begin{proof}
  See Young~\cite[Lemma 7.1]{young2017}. 
\end{proof}

\begin{lemma}\label{lemma: H-}
  Let $H^-$ be given by \eqref{eqn: H}.
  There exists a function $g$ depending on $T$ and $M$ satisfying $g^{(j)}(y)\ll_{j,A} (1+|y|)^{-A}$,
  so that
  \begin{equation}\label{eqn: H^-=}
    H^-(x) = MT\int_{|v|\leq \frac{M^\varepsilon}{M}}\cos(x\sinh(v))e\left(\frac{vT}{\pi}\right) g(Mv) \dd v + O(T^{-A}).
  \end{equation}
  Furthermore, $H^-(x)\ll (x+T)^{-A}$ unless $x\asymp T$.
  In addition, for $x\ll T,$ we have $H^{-}(x)\ll_\delta M T^\delta x^{1-\delta}$ for any $\delta\in(0,1)$.
\end{lemma}

\begin{proof}
  See Young~\cite[Lemma 7.2]{young2017}.
\end{proof}

\subsection{The spectral large sieve inequality}

We recall the following spectral large sieve inequality of Iwaniec \cite{iwaniec1992spectral},
Luo \cite{luo1995spectral}, and Jutila \cite{jutila2000spectral}.

\begin{lemma}\label{lemma:large_sieve}
  We  have
  \[
    \sum_{T<t_j \leq T+1}  \Big| \sum_{n\leq N} a_n  \lambda_j(n) \Big|^2
    \ll (N+T)^{1+\varepsilon} \sum_{n\leq N} |a_n|^2,
  \]
  for any complex sequence $\{a_n\}$.
\end{lemma}


\subsection{Stirling's formula}


For fixed $\sigma\in\mathbb{R}$, real $|t|\geq10$ and any $J>0$, we have Stirling's formula
\begin{equation}\label{eqn:Stirling_J}
  \Gamma(\sigma+it) = e^{-\frac{\pi}{2}|t|} |t|^{\sigma-\frac{1}{2}} \exp\left( it\log\frac{|t|}{e} \right) \left( g_{\sigma,J}(t) + O_{\sigma,J}(|t|^{-J}) \right),
\end{equation}
where
\[
  t^j \frac{\partial^j}{\partial t^j} g_{\sigma,J}(t) \ll_{j,\sigma,J} 1
\]
for all fixed $j\in \mathbb{N}_0$.
More precisely, we have
\begin{equation}\label{eqn:Stirling_7}
  \log  \Gamma(z) = z\log z + \frac{1}{2} \log \frac{2\pi}{z} + \frac{1}{12 z} -\frac{1}{360z^3} + \frac{1}{1260z^5} + O(|z|^{-7}).
\end{equation}

\subsection{Oscillatory integrals}

Let $\mathcal{F}$ be an index set and $X=X_T:\mathcal{F}\rightarrow \mathbb{R}_{\geq1}$ be a function of $T\in\mathcal{F}$. A family of $\{w_T\}_{T\in\mathcal{F}}$ of smooth functions supported on a product of dyadic intervals in $\mathbb{R}_{>0}^d$ is called \emph{$X$-inert} if for each $j=(j_1,\ldots,j_d) \in \mathbb{Z}_{\geq0}^d$ we have
\[
  \sup_{T\in\mathcal{F}} \sup_{(x_1,\ldots,x_d) \in \mathbb{R}_{>0}^d}
  X_T^{-j_1-\cdots -j_d} \left| x_1^{j_1} \cdots x_d^{j_d} w_T^{(j_1,\ldots,j_d)} (x_1,\ldots,x_d) \right|
   \ll_{j_1,\ldots,j_d} 1.
\]


We will use the following integration by parts and stationary phase lemmas several times.

\begin{lemma}\label{lemma:repeated_integration_by_parts}
  Let $Y\geq1$. Let $X,\; V,\; R,\; Q>0$ and suppose that $w=w_T$ is a smooth function with  $\supp w \subseteq [\alpha,\beta]$ satisfying $w^{(j)}(\xi) \ll_j X V^{-j}$ for all $j\geq0$.
  Suppose that on the support of $w$, $h=h_T$ is smooth and satisfies that
  $h'(\xi)\gg R$ and $ h^{(j)}(\xi) \ll Y Q^{-j}$, for all $j\geq2.$
  Then for arbitrarily large $A$ we have
    \[
      I = \int_{\mathbb{R}} w(\xi) e(h(\xi))  \dd \xi  \ll_A (\beta-\alpha)  X \left[  \left(\frac{QR}{\sqrt{Y}}\right)^{-A} + (RV)^{-A}  \right].
    \]
\end{lemma}

\begin{proof}
  See \cite[Lemma 8.1]{BlomerKhanYoung2013distribution}.
\end{proof}

\begin{lemma}\label{lemma:stationary_phase}
  Suppose $w_T$ is $X$-inert in $t_1,\ldots,t_d$, supported on  $t_i\asymp X_i$ for $i=1,2,\ldots,d$. Suppose that on the support of $w_T$, $h=h_T$ satisfies that
  \[
    \frac{\partial^{a_1+a_2+\cdots +a_d}}{\partial t_1^{a_1}\cdots \partial t_d^{a_d}} h(t_1,t_2,\ldots,t_d) \ll_{a_1,\ldots,a_d}  \frac{Y}{X_1^{a_1} X_2^{a_2}\cdots X_d^{a_d}},
  \]
  for all $a_1,\ldots,a_d\in \mathbb{Z}_{\geq0}$. Let
  \[
    I = \int_{\mathbb{R}} w_T(t_1,t_2,\ldots,t_d) e^{i h(t_1,t_2,\ldots,t_d)}  \dd t_1.
  \]
   Suppose $\frac{\partial^{2}}{\partial t_1^{2}} h(t_1,t_2,\ldots,t_d) \gg \frac{Y}{X_1^2}$
  for all $(t_1,t_2,\ldots,t_d)\in \supp w_T$, and there exists $t_0 \in\mathbb{R}$ such that $ \frac{\partial}{\partial t_1} h(t_0,t_2,\ldots,t_d)=0$.
  Suppose that $Y/X^2 \geq R \geq 1$. Then
  \[
    I
    = \frac{X_1}{\sqrt{Y}} e^{i h(t_0,t_2,\ldots,t_d)} W_T(t_2,\ldots,t_d) + O_A(X_1 R^{-A}),
  \]
  for some $X$-inert family of functions $W_T$ and any $A>0$.
\end{lemma}

\begin{proof}
  See \cite[\S 8]{BlomerKhanYoung2013distribution} and \cite[\S 3]{KPY}.
\end{proof}

In the applications of Lemma \ref{lemma:stationary_phase}, we will explicitly give estimates of the derivatives for the first variable. For other derivatives we will also check all those conditions, but may not write them down explicitly.


\section{First moment of $\GL(3)\times \GL(2)$ $L$-functions}
\label{sec:first_moment_L_functions}

In this section, we will prove Theorem \ref{thm:moment_L_functions_6}.

\subsection{Approximate functional equation}

Let $\phi$ be a Hecke--Maass cusp form with the spectral parameter $T>0$.
The symmetric square lift $\Sym^2 \phi$ is a  Hecke--Maass cusp form for $\SL(3,\mathbb{Z})$ with Fourier coefficients $A(m,n)$.
The $\GL(3)$ $L$-function is defined as
\[
  L(s,\Sym^2 \phi) = \sum_{n\geq1} \frac{A(1,n)}{n^s}, \quad \Re(s)>1.
\]
Let $\phi_j$ be an even Hecke--Maass cusp form with the spectral parameter $t_j>0$ and Fourier coefficients $\lambda_j(n)$.
The $\GL(3)\times \GL(2)$ Rankin--Selberg $L$-function is defined as
\[
  L(s,\phi_j\times \Sym^2 \phi) = \sum_{m\geq1}\sum_{n\geq1} \frac{A(m,n)\lambda_j(n)}{(m^2n)^s}, \quad \Re(s)>1.
\]
The functional equation of $L(s,\phi_j \times \Sym^2 \phi)$ is
\[
  \Lambda(s,\phi_j \times \Sym^2 \phi)
  := L_\infty(s,\phi_j \times \Sym^2 \phi)
  L(s,\phi_j \times \Sym^2 \phi)
  = \Lambda(1-s,\phi_j \times \Sym^2 \phi)
\]
where
\[
  L_\infty(s,\phi_j \times \Sym^2 \phi) = \pi^{-3s} \prod_{\pm}
  \Gamma\left(\frac{s\pm it_j}{2}\right)
  \Gamma\left(\frac{s+2iT\pm it_j}{2}\right)
  \Gamma\left(\frac{s-2iT\pm it_j}{2}\right).
\]
Let $E_t$ be the Eisenstein series and $t\in \mathbb{R}$. We have
\[
  L(s+it, \Sym^2 \phi)L(s-it, \Sym^2 \phi)  = L(s,E_t \times \Sym^2 \phi),
\]
and  the functional equation
\[
  \Lambda(s,E_t \times \Sym^2 \phi)
  := L_\infty(s,E_t \times \Sym^2 \phi)
  L(s,E_t \times \Sym^2 \phi)
  = \Lambda(1-s,E_t \times \Sym^2 \phi)
\]
with
\[
  L_\infty(s,E_t \times \Sym^2 \phi) = \pi^{-3s} \prod_{\pm}
  \Gamma\left(\frac{s\pm it}{2}\right)
  \Gamma\left(\frac{s+2iT\pm it}{2}\right)
  \Gamma\left(\frac{s-2iT\pm it}{2}\right).
\]
Note that
\[
  |L(1/2+it, \Sym^2 \phi)|^2 = L(1/2,E_t \times \Sym^2 \phi).
\]
We have the following approximate functional equation.

\begin{lemma} \label{lemma:AFE6}
  With the notation as above.
  Then we have
  \[
     L(1/2,\phi_j \times \Sym^2 \phi)
     = 2 \sum_{m,n\geq1}  \frac{A(m,n)\lambda_j(n)}{m n^{1/2}} V_{t_j}\left(m^2 n\right),
  \]
  and
  \[
     |L(1/2+it, \Sym^2 \phi)|^2
     = 2 \sum_{m,n\geq1}  \frac{A(m,n)\eta_t(n)}{m n^{1/2}} V_{t}\left(m^2 n\right),
  \]
  where
  \[
    V_{t_j}\left(y \right) = \frac{1}{2\pi i} \int_{(2)} \frac{L_\infty(1/2+s,\phi_j \times \Sym^2 \phi) }{L_\infty(1/2,\phi_j \times \Sym^2 \phi) } y^{-s} G(s) \frac{\dd s}{s},
  \]
  and $V_t(y)$ is defined as above with $E_t$ replacing $\phi_j$. Here $G(s)=e^{s^2}$.
\end{lemma}

\begin{proof}
 This is \cite[Theorem  5.3]{IwaniecKowalski2004analytic}.
\end{proof}

\begin{lemma}
   Assume $|t_j-T|\leq T^{1/2+\varepsilon}$. Then we have
  \[ y^k V_{t_j}^{(k)}(y) \ll_{k,A} \left( 1+ \frac{y}{T^3} \right)^{-A},\]
  for any $A>0$.
  Let $U=(\log T)^2$. Then we have
  \[
    V_{t_j}(y) = \sum_{k=0}^{6} \sum_{l=0}^{12} T^{-k} \left(\frac{t_j^2-T^2}{T^2}\right)^{l} V_{k,l}\left(\frac{y}{T^3}\right)
    + O\left(y^{-\varepsilon}T^\varepsilon e^{-U} + T^{-5+\varepsilon}\right),
  \]
  where
  \[
    V_{k,l}\left(\frac{y}{T^3}\right) = \frac{1}{2\pi i}
    \int_{\varepsilon-iU}^{\varepsilon+iU} Q_{k,l}(s)  \left(\frac{3}{8 \pi^3}\right)^s y^{-s} G(s) \frac{\dd s}{s}.
  \]
\end{lemma}

\begin{proof}
The first claim is  \cite[Proposition 5.4]{IwaniecKowalski2004analytic}.
The proof of the second claim is similar to \cite[\S5]{young2017} and \cite[Lemma 2.2]{huang2021hybrid}.
By Stirling's formula we have
\begin{equation}\label{eqn:V8U}
   V_{t_j}(y) = \frac{1}{2\pi i} \int_{\varepsilon-iU}^{\varepsilon+iU}
   \frac{L_\infty(1/2+s,\phi_j \times \Sym^2 \phi) }{L_\infty(1/2,\phi_j \times \Sym^2 \phi) }
  y^{-s} G(s) \frac{\dd s}{s}
  + O(y^{-\varepsilon}T^\varepsilon e^{-U}).
\end{equation}
Let $t\asymp T$ and $t>0$.
Since $|s|\ll U$, by Stirling's formula \eqref{eqn:Stirling_7} we have
\begin{align*}
   \prod_{\pm} \frac{  \Gamma\left(\frac{1/2+s\pm it}{2}\right)}
   { \Gamma\left(\frac{1/2\pm it}{2}\right)}  &
  = \left( \frac{t^2}{4} \right)^{s/2} \left( 1 + \frac{P_{2}(s)}{t^{2}}+ \frac{P_{4}(s)}{t^{4}} + O\left( T^{-5} \right) \right),
\end{align*}
where $P_{2}(s)=(3 + 10 s + 3 s^2 - 4 s^3)/12 $ and $P_4(s)=(-48 s - 580 s^2 - 660 s^3 - 175 s^4 + 168 s^5 + 80 s^6) /1440$ are polynomials.
Hence we have
\begin{multline*}
  \frac{L_\infty(1/2+s,\phi_j \times \Sym^2 \phi) }{L_\infty(1/2,\phi_j \times \Sym^2 \phi) }
=  \pi^{-3s}  \left( \frac{t_j^2}{4} \right)^{s/2}
 \left( \frac{4T^2-t_j^2}{4} \right)^{s}
 \left( 1 + \sum_{k=1}^{2}\frac{P_{2k}(s)}{t_j^{2k}}    \right)
 \\ \cdot
 \left( 1 + \sum_{k=1}^{2}\frac{P_{2k}(s)}{(2T+t_j)^{2k}}    \right)
 \left( 1 + \sum_{k=1}^{2}\frac{P_{2k}(s)}{(2T-t_j)^{2k}}    \right)
 \left( 1  + O\left( T^{-5} \right) \right) .
 \end{multline*}
 Since $|t_j-T|\leq T^{1/2+\varepsilon}$, we have
 \begin{multline*}
 \left( 1 + \sum_{k=1}^{2}\frac{P_{2k}(s)}{(2T+t_j)^{2k}}    \right)
 \left( 1 + \sum_{k=1}^{2}\frac{P_{2k}(s)}{(2T-t_j)^{2k}}    \right)
 \\
 = \left( 1
 +  \frac{2 P_{2}(s) (4T^2+t_j^2) + P_{2}(s)^2 }{(4T^2-t_j^2)^{2}}
 +  \frac{2 P_{4}(s) (16 T^4 + 24 t_j^2 T^2 +t_j^4)}{(4T^2-t_j^2)^{4}}     \right)
 \left( 1  + O\left( T^{-5}\right) \right) .
\end{multline*}
Note that $t_j^2 = T^2 \left(1+\frac{t_j^2-T^2}{T^2}\right)$, $\Re(s)=\varepsilon$ and $|\Im(s)|\leq U$. We have
\begin{multline*}
   \left( \frac{t_j^2}{4} \right)^{s/2}
   = \exp\left(\frac{s}{2} \log \frac{T^2}{4} + \frac{s}{2} \log \left(1+\frac{t_j^2-T^2}{T^2}\right) \right) \\
   = \left( \frac{T}{2} \right)^{s}  \sum_{l=0}^{20} Q_{l,1}(s) \left(\frac{t_j^2-T^2}{T^2}\right)^l + O(T^{-9}),
\end{multline*}
\begin{multline*}
   \left( \frac{4T^2-t_j^2}{4} \right)^{s}
   = \exp\left( s \log \frac{3T^2}{4} + s \log \left(1+\frac{t_j^2-T^2}{3T^2}\right) \right) \\
   = \left( \frac{3T^2}{4} \right)^{s}  \sum_{l=0}^{20} Q_{l,2}(s) \left(\frac{t_j^2-T^2}{T^2}\right)^l + O(T^{-9}),
\end{multline*}
\begin{equation*}
   \frac{1}{t_j^{2k}}
   = \frac{1}{T^{2k} \left(1+\frac{t_j^2-T^2}{T^2}\right)^k} \\
   = \frac{1}{T^{2k}}  \sum_{l=0}^{20} Q_{k,l} \left(\frac{t_j^2-T^2}{T^2}\right)^l + O(T^{-9}),
\end{equation*}
\begin{equation*}
    \frac{1}{(4T^2-t_j^2)^{k}}
   = \frac{1}{3^{k}T^{2k} \left(1+\frac{t_j^2-T^2}{3T^2}\right)^{k} } \\
   = \frac{1}{3^{k} T^{2k}}  \sum_{l=0}^{20} Q_{k,l} \left(\frac{t_j^2-T^2}{3T^2}\right)^l + O(T^{-9}),
\end{equation*}
where $Q_{l,1}(s)$ and $Q_{l,2}(s)$ are certain polynomials of degree $\leq l$, and $Q_{k,l}$ are constants.  Hence we obtain
\begin{equation*}
  \frac{L_\infty(1/2+s,\phi_j \times \Sym^2 \phi) }{L_\infty(1/2,\phi_j \times \Sym^2 \phi) }
  =  \left(\frac{3 T^3}{8 \pi^3}\right)^s \sum_{k=0}^{6}
  \sum_{\ell=0}^{12}  Q_{k,l}(s) \frac{1}{ T^{k} } \frac{(t_j^2-T^2)^\ell}{T^{2\ell}} + O(T^{-5}),
\end{equation*}
where $Q_{k,l}(s)$ are certain polynomials.
By \eqref{eqn:V8U} we complete the proof of the lemma.
\end{proof}

\subsection{Applying the Kuznetsov trace formula}
Define
\[
  \mathcal{S} = \sideset{}{'}\sum_{t_j} h_{T,M}(t_j) \omega_j L(1/2,\phi_j\times \Sym^2 \phi) + \frac{1}{4\pi} \int_{\mathbb{R}} h_{T,M}(t) \omega(t) |L(1/2+it,\Sym^2 \phi)|^2 \dd t,
\]
where $'$ means summing over even forms and
\[
  h_{T,M}(t) = \exp\left( - \frac{(t-T)^2}{M^2}\right)+ \exp\left( - \frac{(t+T)^2}{M^2}\right).
\]
Note that $L(1/2,\phi_j\times \Sym^2 \phi)=0$ for any odd $\phi_j$.
By Lemma \ref{lemma:AFE6} we get
\begin{equation}\label{eqn:S2Skl}
  \mathcal{S} = 2 \sum_{k=0}^{6} \sum_{l=0}^{12} T^{-k} \mathcal{S}_{k,l} + O(T^{-5/2+\varepsilon} M),
\end{equation}
where
\begin{multline*}
  \mathcal{S}_{k,l} = \sum_{m,n\geq1} \frac{A(m,n)}{m n^{1/2}} V_{k,l}\left(\frac{m^2 n}{T^3} \right) \Big(
  \sideset{}{'}\sum_{t_j} h_{T,M}(t_j)\left(\frac{t_j^2-T^2}{T^2}\right)^{l}  \omega_j \lambda_j(n)
  \\
  + \frac{1}{4\pi} \int_{\mathbb{R}} h_{T,M}(t)\left(\frac{t^2-T^2}{T^2}\right)^{l}  \omega(t) \eta_t(n) \dd t  \Big).
\end{multline*}
We only consider the case $l=0$, since we save additional powers of $T$ when $l\geq1$.
By Lemma \ref{lemma: KTF}, we get
\begin{equation}\label{eqn:Skl=D+R}
  \mathcal{S}_{k,l} =
  \mathcal{D} + \sum_{\pm} \mathcal{R}^\pm,
\end{equation}
where the diagonal term is
\[
  \mathcal{D} = \sum_{m,n\geq1} \frac{A(m,n)}{m n^{1/2}} V_{k,l}\left(\frac{m^2 n}{T^3} \right) \frac{1}{2}\delta_{1,n}H
\]
and the off-diagonal terms are
\begin{equation*}
  \mathcal{R}^{\pm} = \frac{1}{2}\sum_{m,n\geq1} \frac{A(m,n)}{m n^{1/2}} V_{k,l}\left(\frac{m^2 n}{T^3} \right) \sum_{c\geq1}\frac{S(n,\pm 1;c)}{c} H^{\pm}\left(\frac{4\pi\sqrt{n}}{c}\right) .
\end{equation*}

By \cite{GelbartJacquet} we know that $\Sym^2\phi$ is automorphic. By the Rankin--Selberg theory we have (see e.g.\ \cite{li2010upper})
\begin{equation}\label{eqn:RS3}
  \sum_{m^2n\leq x} |A(m,n)|^2 \ll (Tx)^\varepsilon x.
\end{equation}
Note that $H\ll TM$.
Hence we have
\begin{equation}\label{eqn:D}
   \mathcal{D} \ll T^{1+\varepsilon } M .
\end{equation}

Inserting a smooth partition of unity we get
\begin{equation}\label{eqn:R<<R(N)}
    \mathcal{R}^{\pm} \ll \log T \cdot \sup_{1\leq N \leq T^{3+\varepsilon}} |\mathcal{R}^{\pm}(N)| +  O(T^{-2022}) ,
\end{equation}
where
\begin{equation*}
  \mathcal{R}^{\pm}(N)
  = \frac{1}{N^{1/2}} \sum_{m,n\geq1}  A(m,n)  V \left(\frac{m^2 n}{N} \right)
  \sum_{ c\geq 1}\frac{S(n,\pm 1;c)}{c} H^{\pm}\left(\frac{4\pi\sqrt{n}}{c}\right),
\end{equation*}
and $V$ is an $1$-inert function with $\supp V \subseteq [1,2]$.
Theorem \ref{thm:moment_L_functions_6} will follow upon proving that $\mathcal{R}^{\pm}(N) \ll T^{3/2+\varepsilon}/M^{1/2}$ for
$T^{1/3+\varepsilon}\leq M\leq T^{1/2-\varepsilon}$.

\subsection{Applying the Voronoi summation formula}
We first consider $\mathcal{R}^+(N)$.
By Lemma \ref{lemma: H+}, we have
\begin{multline*}
  \mathcal{R}^{+}(N) = \frac{1}{N^{1/2}} \sum_{m\geq1}\sum_{n\geq1} A(m,n)  V\left(\frac{m^2 n}{N} \right) \sum_{1\leq c\ll \frac{\sqrt{N}}{T^{1-\varepsilon}Mm} }\frac{S(n, 1;c)}{c}  \\
  \cdot
  MT\int_{|v|\leq \frac{M^\varepsilon}{M}}\cos\left(\frac{4\pi\sqrt{n}}{c}\cosh(v)\right) e\left(\frac{vT}{\pi}\right)g(Mv) \dd v + O(T^{-2022}).
\end{multline*}
Making a change of variable $cm=r$, we get
\begin{multline*}
  \mathcal{R}^{+}(N)
  = \frac{MT}{2 N^{1/2}}  \int_{|v|\leq \frac{M^\varepsilon}{M}} \sum_{1\leq r\ll \frac{\sqrt{N}}{T^{1-\varepsilon}M} }\frac{1}{r}  \sum_{m\mid r} \sum_{n\geq1} m A(m,n)  S(n, 1;r/m) V\left(\frac{m^2 n}{N} \right)  \\
  \cdot \sum_{\sigma_1=\pm }
  e\left(\sigma_1\frac{2m\sqrt{n}}{r}\cosh(v)\right)
  e\left(\frac{vT}{\pi}\right)g(Mv)\dd v + O(T^{-2022}) .
\end{multline*}

Let $w$ be a smooth compactly supported function on $(0,\infty)$,
and let $\tilde{w}(s):=\int_{0}^{\infty} w(x)x^s\frac{\dd x}{x}$
be its Mellin transform.
Let
\begin{equation}\label{eqn:gamma3}
  L_\infty(s,\Sym^2\phi) = \pi^{-3s/2} \Gamma\left(\frac{s-2iT}{2}\right) \Gamma\left(\frac{s}{2}\right)\Gamma\left(\frac{s+2iT}{2}\right).
\end{equation}
We have the following Voronoi summation formula of Miller--Zhou \cite{MZ}.
\begin{lemma}\label{lemma:VSF}
  For $w\in C_c^\infty(0,\infty)$ define its Hankel transform $W$ by
  \begin{equation}\label{eqn:W+-}
    W^\pm (y) = \frac{1}{2\pi i} \int_{(-3)} G^\pm(s) \tilde{w}(s) y^{s-1} \dd s,
  \end{equation}
  where $\tilde{w}(s)$ is the Mellin transform of $w$,  and
  \[
    G^\pm (s) = \frac{L_\infty(1-s,\Sym^2\phi)}{L_\infty(s,\Sym^2\phi)} \pm i^{-3} \frac{L_\infty(2-s,\Sym^2\phi)}{L_\infty(1+s,\Sym^2\phi)}.
  \]
  Let $a,\bar{a},r$ be integers with $a\bar{a}\equiv1 \pmod{r}$ and $r>0$. Then we have
  \begin{equation*}
    \sum_{n_2\mid r} \sum_{n_1\geq1} n_2 A(n_1,n_2) S(n_1,\bar{a},r/n_2) w(n_1n_2^2) \\
    = \sum_{\pm } \sum_{n\geq1} \frac{A(1,n)}{r} e\left(\pm \frac{an}{r}\right) W^\pm \left( \frac{n}{r^3} \right).
  \end{equation*}
\end{lemma}

We need the following sign sensitive expression of $G^{\pm}(s)$, which is due to Blomer \cite[P. 1397]{blomer2012subconvexity}. Note that there is a typo. The following lemma follows from the initial identity \eqref{eqn:gamma3} coupled with standard applications of the Legendre duplication formula and Euler's reflection formula.

\begin{lemma}\label{lemma:G}
  We have
  \begin{multline*}
    G^\pm (s)
    =  \frac{ \pi^{-3+3s}  i }{   2^{ 2 - 3s }   }
   \Gamma( 1-s+2iT )\Gamma( 1-s )\Gamma( 1-s-2iT )
   \\ \cdot
    \left(   e\Big(\frac{\pm 3s}{4}\Big)
    + e\Big(\frac{\mp s}{4}\Big) \Big(1+e(iT)+e(-iT)\Big) \right).
  \end{multline*}
\end{lemma}

By Lemma \ref{lemma:VSF}, we get
\begin{align} \label{eqn:R+<<I}
  \mathcal{R}^{+}(N)
   & = \frac{MT}{2N^{1/2}}  \int_{|v|\leq \frac{M^\varepsilon}{M}}
  \sum_{1\leq r\ll \frac{\sqrt{N}}{T^{1-\varepsilon}M} }\frac{1}{r} \sum_{\sigma_1=\pm }
  \nonumber
  \\
  & \hskip 1.5cm \cdot
  \sum_{\pm } \sum_{n\geq1} \frac{A(1,n)}{r} e\left(\pm \frac{n}{r}\right)
  W_{\sigma_1,v}^\pm \left( \frac{n}{r^3} \right)
  e\left(\frac{vT}{\pi}\right)g(Mv)\dd v   + O(T^{-2022}) \nonumber \\
  & = \frac{MT}{2N^{1/2}}
  \sum_{1\leq r\ll \frac{\sqrt{N}}{T^{1-\varepsilon}M} }\frac{1}{r} \sum_{\sigma_1=\pm }
  \sum_{\pm } \sum_{n\geq1} \frac{A(1,n)}{r} e\left(\pm \frac{n}{r}\right)  I_{\sigma_1}^\pm(n,r)
   + O(T^{-2022}) ,
\end{align}
where $ W_{\sigma_1,v}^\pm$ is given by \eqref{eqn:W+-} with $w_{\sigma_1,v}(y)=  V\left(\frac{y}{N} \right)
  e\left(\sigma_1\frac{2\sqrt{y}}{r}\cosh(v)\right) $, and
\[
  I_{\sigma_1}^\pm(n,r) := \int_{|v|\leq \frac{M^\varepsilon}{M}} W_{\sigma_1,v}^\pm \left( \frac{n}{r^3} \right)
  e\left(\frac{vT}{\pi}\right)g(Mv)\dd v .
\]

\subsection{Integral transforms} \label{subsec:IT}

Now we need to estimate $I_{\sigma_1}^\pm(n,r)$. We will first deal with  $W_{\sigma_1,v}^\pm \left( \frac{n}{r^3} \right)$ in the following lemma.
\begin{lemma}\label{lemma:W}
  Let $ W_{b}^\pm$ be given by \eqref{eqn:W+-} with $w_{b}(x)=  V\left(\frac{x}{N} \right)
  e\left(b \sqrt{\frac{x}{N}}\right) $. Assume $b\in\mathbb{R}$  and $|b|\asymp B=\frac{\sqrt{N}}{r} \gg T^{1-\varepsilon} M$.
  \begin{itemize}
     \item [i)]  Assume $0< y\leq T^{2022}$. Then we have  $W_{b}^\pm (y)=O(T^{-A})$ unless $\sgn(b)=\pm$ and $Ny\asymp B^3$, in which case we have
  \[
    W_{b}^\pm (y)
    = N^{1/2}y^{-1/2}
      V_2\left(\xi_0,\frac{\pm b}{B}\right) e^{ih_2(\xi_0)} + O(T^{-A}),
  \]
  for any positive constant $A$ and some 1-inert function $V_2$ depending on $A$. Here
  $\xi_0 = \pm \frac{4 \pi Ny}{b^3} + \sqrt{(\frac{4 \pi Ny}{b^3})^2+4 \frac{T^2}{b^2}}$
  and
  $
    h_2(\xi_0) = -b \xi_0 - 2T \log \frac{|b| \xi_0 \mp 2 T}{|b| \xi_0 \pm 2 T}.
  $
    \item [ii)] For $y>0$, we have
  \[
    W_{b}^\pm (y) \ll y^{-A-1} B^{3A+3} N^{-A},
  \]
  for any positive constant $A$.
  \end{itemize}
\end{lemma}

\begin{proof}
  By \eqref{eqn:W+-} and moving the line of integration to $\Re(s)=1/2$, we have
  \begin{equation*}
    W_b^\pm (y) = \frac{1}{2\pi} \int_{\mathbb{R}} G^\pm(1/2+i\tau) \tilde{w}_{b}(1/2+i\tau) y^{-1/2+i\tau} \dd \tau.
  \end{equation*}
  We first deal with $\tilde{w}_{b}$. Since $w_{b}(x)=  V\left(\frac{x}{N} \right)
  e\left(b \sqrt{\frac{x}{N}}\right) $, we have
  \[
    \tilde{w}_{b}(1/2+i\tau) = \int_{0}^{\infty} V\left(\frac{x}{N} \right)
  e\left(b\sqrt{\frac{x}{N}}\right)  x^{1/2+i\tau-1} \dd x.
  \]
  Making a change of variable $x=N\xi^2$, we get
  \[
    \tilde{w}_{b}(1/2+i\tau)
    = N^{1/2+i\tau} \int_{0}^{\infty} 2 V\left(\xi^2\right)
    e\left(b\xi + \frac{\tau}{\pi} \log \xi\right)  \dd \xi.
  \]
  Let
  \[
    h_1(\xi) = b\xi + \frac{\tau}{\pi} \log \xi.
  \]
  Then we have
  \[
    h_1'(\xi) = b + \frac{\tau}{\pi  \xi} ,
  \]
  and
  \[
    h_1''(\xi)= -\frac{\tau}{\pi \xi^2}, \quad
    h_1^{(j)}(\xi)\asymp_j |\tau|, \quad j\geq2.
  \]
  If $\sgn(\tau)=\sgn(b)$ or $|\tau|\geq |b|/2023$ or $|\tau|\leq 2023|b|$, then we have $|h_1'(\xi)| \gg |b|+|\tau|$. By Lemma \ref{lemma:repeated_integration_by_parts},
  we have
  \[
    \tilde{w}_{b}(1/2+i\tau) \ll_A (T(1+|\tau|))^{-A},
  \]
  for any $A>0$.
  Now assume   $\sgn(\tau)=-\sgn(b)$ and $|\tau|\asymp |b|$.
  The solution of $h_1'(\xi)=0$ is
  \[
    \xi_0 = -\frac{\tau }{\pi b} .
  \]
  We have  $h_1(\xi_0)=\frac{\tau}{\pi} \log  \frac{-\tau  }{ \pi e b }$.
  Hence by Lemma \ref{lemma:stationary_phase},
  we have
  \begin{equation*}
     \tilde{w}_{b}(1/2+i\tau)
     = \frac{ N^{1/2+i\tau}}{\sqrt{|\tau|}}
     e\left( \frac{\tau}{\pi} \log  \frac{-\tau  }{ \pi e b } \right)
     V_1\left( \frac{\tau}{-b} , \frac{|b|}{B} \right) \\
      + O((|\tau|+T)^{-A}),
  \end{equation*}
  for any positive $A$ and some 1-inert function $V_1$ which depends on $A$ with compact support.
  Note that we have
  \[|\tau| \asymp B \gg T^{1-\varepsilon} M.\]


  By Lemma \ref{lemma:G}, we have
  \begin{multline*}
    G^\pm(1/2+i\tau) =  \frac{ \pi^{-3/2+3i\tau}  i}{   2^{ 1/2 - 3i\tau }   }
   \Gamma( 1/2-i\tau+2iT )\Gamma(1/2-i\tau)\Gamma(1/2-i\tau-2iT ) \\
   \cdot \left(   e\Big(\frac{\pm 3/2 \pm 3i\tau}{4}\Big)
  + e\Big(\frac{\mp 1/2\mp i\tau}{4}\Big) \Big(1+e(iT)+e(-iT)\Big) \right).
  \end{multline*}
  By Stirling's formula \eqref{eqn:Stirling_J}, we have
  \begin{multline}\label{eqn:G==}
    G^\pm(1/2+i\tau) =
    \exp\bigg( -i(\tau+2T)\log \frac{|\tau+2T|}{e} -i(\tau-2T)\log \frac{|\tau-2T|}{e} \\
    - i\tau \log \frac{|\tau|}{(2\pi)^{3} e} \bigg)
    \exp\left(\mp \frac{3\pi \tau}{2} -\frac{3\pi |\tau|}{2}\right)
    \left( w_{J}(\tau) + O_{J}(|\tau|^{-J}) \right)
    + O\left( \exp\Big(-\frac{\pi |\tau|}{2} \Big) \right) ,
  \end{multline}
  for some large $J>0$ and some weight function $w_J$ satisfying that  $\tau^j \frac{\partial^j}{\partial \tau^j} w_{J}(\tau) \ll_{j,J} 1$.
  If $\sgn(\tau)=\pm$, then we have $G^\pm(1/2+i\tau) \ll \exp\big(-\frac{\pi |\tau|}{2} \big) $ and the contribution to $W_b^\pm (y)$ is
  $O(T^{-A})$ for any $A>0$.

  If $\sgn(\tau)=\mp=-\sgn(b)$, then we get
  \begin{multline*}
    W_b^\pm (y) = \frac{1}{2\pi } \int_{\mathbb{R}} \pi^{3i\tau}  \exp\left( -i(\tau+2T)\log \frac{|\tau+2T|}{2e} - i\tau \log \frac{|\tau|}{2e} -i(\tau-2T)\log \frac{|\tau-2T|}{2e} \right) \\
    \cdot
    \left( w_{J}(\tau) + O_{J}(|\tau|^{-J}) \right)
    \frac{ N^{1/2+i\tau}}{\sqrt{|\tau|}}
    e\left( \frac{\tau}{\pi} \log  \frac{-\tau  }{ \pi e b } \right)
     V_1\left( \frac{\tau}{-b} , \frac{\pm b}{B} \right)
     y^{-1/2+i\tau} \dd \tau
     + O(T^{-A})\\
      = N^{1/2}y^{-1/2}
      \frac{1}{2\pi }
      \int_{\mathbb{R}}   w_{J}(\tau)
    \frac{ 1}{\sqrt{|\tau|}}
     V_1\left( \frac{\tau}{-b} , \frac{\pm b}{B} \right)
     (\pi^{3}Ny)^{i\tau}
    e\left( \frac{\tau}{\pi} \log  \frac{-\tau}{\pi e b} \right) \\
    \cdot \exp\left( -i(\tau+2T)\log \frac{|\tau+2T|}{2e} - i\tau \log \frac{|\tau|}{2e} -i(\tau-2T)\log \frac{|\tau-2T|}{2e} \right)
      \dd \tau + O(T^{-A}).
  \end{multline*}
  Making a change of variable $\tau=-b \xi=\mp |b|\xi$, we get
  \begin{multline*}
    W_b^\pm (y)
      =   N^{1/2}y^{-1/2}
      \frac{\mp \sqrt{|b|}}{2\pi }
      \int_{\mathbb{R}}   w_{J}(-b\xi)
    \frac{ 1}{\sqrt{\xi}} V_1\left( \xi,\frac{|b|}{B} \right)
    \\
    \cdot
    \exp\Big( i(b \xi-2T)\log \frac{|b \xi-2T|}{2e}
    + i b \xi \log \frac{|b|\xi}{2e}
    + i(b \xi+2T)\log \frac{| b \xi+2 T|}{2e} \\
    -ib\xi\log (\pi^{3}Ny)
     -  i 2b \xi \log  \frac{\xi}{ \pi e } \Big)
      \dd \xi + O(T^{-A}).
  \end{multline*}
  Let
  \begin{multline*}
    h_2(\xi) = (b \xi-2T)\log \frac{|b| \xi \pm 2T}{2e}
    + b \xi \log \frac{|b|\xi}{2e}
    + (b\xi+2T)\log \frac{ |b| \xi \mp 2 T}{2e}
     \\
    - b \xi\log (\pi^{3}Ny)
    - 2 b \xi \log  \frac{\xi}{ \pi e } \\
    = b \xi \log \frac{ (|b| \xi \pm 2T) |b|\xi ( |b| \xi \mp 2 T) }
    {8 \pi e Ny \xi^2 }
    - 2T \log \frac{|b| \xi \mp 2 T}{|b| \xi \pm 2 T}.
  \end{multline*}
  Then we have
  \begin{multline*}
    h_2'(\xi)= b \log \frac{|b| \xi \pm 2T}{2}
    + b \log \frac{ |b|\xi}{2}
    + b \log \frac{ |b| \xi \mp 2 T}{2}
    - b \log (\pi^{3}Ny)
    - 2b \log \frac{ \xi}{\pi}
    \\
    = b \log \frac{ |b| (|b| \xi \pm 2T)  (|b| \xi \mp 2T) }{ 8 \pi Ny  \xi }.
  \end{multline*}
  and
  \[
    h_2''(\xi) = b \Big( \frac{1}{ \xi-2 T/b} + \frac{1}{ \xi+2T/b} + \frac{1}{\xi}\Big),
    \quad
    h_2^{(j)}(\xi)\asymp_j B, \quad j\geq2.
  \]
  If $Ny \leq B^3/2023$ or $Ny\geq 2023 B^3$, then we have
  $h_2'(\xi) \gg B$. By Lemma \ref{lemma:repeated_integration_by_parts},
  we have
  \[
    W_b^\pm(y) \ll_A T^{-A},
  \]
  for any $A>0$.

  Assume $Ny\asymp B^3$.
  The solution of $h_2'(\xi)=0$, i.e., $\xi^2-\frac{8\pi Ny}{|b|^3}\xi-4T^2/b^2=0$, is
  \begin{equation}\label{eqn:xi0}
    \xi_0 = \pm \frac{4 \pi Ny}{b^3} + \sqrt{\Big(\frac{4 \pi Ny}{b^3}\Big)^2+4 \frac{T^2}{b^2}}.
  \end{equation}
  Note that
  \begin{equation}\label{eqn:h(xi0)}
    h_2(\xi_0) = -b \xi_0 - 2T \log \frac{|b| \xi_0 \mp 2 T}{|b| \xi_0 \pm 2 T}.
  \end{equation}
  Hence by Lemma \ref{lemma:stationary_phase},
  we have
  \begin{equation*}
    W_b^\pm (y)
      = N^{1/2}y^{-1/2}
      V_2\left(\xi_0,\frac{\pm b}{B}\right) e^{ih_2(\xi_0)} + O(T^{-A}),
  \end{equation*}
  for any positive $A$ and some 1-inert function $V_2$ which depends on $A$.

  Now we prove ii). By \eqref{eqn:W+-} and moving the line of integration to $\Re(s)=-A$, we have
  \begin{equation*}
    W_b^\pm (y) = \frac{1}{2\pi} \int_{\mathbb{R}} G^\pm(-A+i\tau) \tilde{w}_{b}(-A+i\tau) y^{-A-1+i\tau} \dd \tau.
  \end{equation*}
  Note that
  \begin{align*}
    \tilde{w}_{b}(-A+i\tau) & = \int_{0}^{\infty} V\left(\frac{x}{N} \right)
  e\left(b \sqrt{\frac{x}{N}}\right) x^{-A+i\tau-1} \dd x \\
  & \leq N^{-A}  \Big| \int_{0}^{\infty} V\left(\xi \right)
  e\left(b \sqrt{\xi}\right)  \xi^{-A+i\tau-1} \dd \xi \Big|.
  \end{align*}
  By the same argument as in the proof of i), we have
  \[
    \tilde{w}_{b}(-A+i\tau) \ll N^{-A} \frac{1}{\sqrt{B}} V(\frac{\tau}{B}) + N^{-A} (|\tau|+T)^{-5A}.
  \]
  By Stirling's formula we get
  \begin{equation*}
    W_b^\pm (y) \ll y^{-A-1} B^{3A+3} N^{-A}.
  \end{equation*}
  This completes the proof.
\end{proof}

\begin{lemma} \label{lemma:I}
  We have $I_{\sigma_1}^\pm(n,r) = O(n^{-5} T^{-A})$ unless $n\asymp N^{1/2}$ and $\pm = \sigma_1$, in which case we have
  \[
    I_{\sigma_1}^\pm(n,r) = r^2
    e\left( \frac{\mp   n}{ r }
      \mp \frac{37  T^2 r  }{12 \pi^2  n}
      \right)
      V_r\left( \frac{n}{\sqrt{N}}\right) + O\left( r^{3/2} M^{-2+\varepsilon} \right),
  \]
  where $V_r$ is certain $(1+ \frac{N^{1/2+\varepsilon}}{rM^4})$ inert function depending on $r$ with compact support.
\end{lemma}

\begin{proof}
  If $n\geq T^{100}$, then by Lemma \ref{lemma:W} (ii) with $b=\sigma_1\frac{2\sqrt{N}}{r}\cosh(v)$, we have
  \[
    W_{\sigma_1,v}^\pm \left( \frac{n}{r^3} \right) \ll
    n^{-A-1} N^{A/2+3/2}  \ll n^{-5} T^{-A},
  \]
  for any large $A>10$. Here we have used the fact $N\ll T^{3+\varepsilon}$.

  If $n\leq T^{100}$, then by Lemma \ref{lemma:W} (i) with $b=\sigma_1\frac{2\sqrt{N}}{r}\cosh(v)$ we have
  $I_{\sigma_1}^\pm(n,r) = O(n^{-5} T^{-A})$
  unless $n\asymp N^{1/2}$ and $\pm = \sigma_1$, in which case we have
  \[
    I_{\sigma_1}^\pm(n,r) = N^{1/2} \Big(\frac{n}{r^3}\Big)^{-1/2} \int_{|v|\leq \frac{M^\varepsilon}{M}} V(\xi_0,\cosh v) e^{ih_2(\xi_0)}
    e\left(\frac{vT}{\pi}\right)g(Mv)\dd v + O(T^{-A}),
  \]
  where
  \[
    \xi_0 = \xi_0(v) =  \frac{ \pi  n}{2 N^{1/2} \cosh^3(v)} + \sqrt{\Big(\frac{ \pi  n}{2 N^{1/2} \cosh^3(v)}\Big)^2 + \frac{T^2 r^2}{N \cosh^2(v)}}
  \]
  and
  \[
    h_2(\xi_0) = -b \xi_0 - 2T \log \frac{1 \mp 2 T/(|b| \xi_0)}{1 \pm 2 T/(|b| \xi_0)}.
  \]
  By  $ \xi_0^2 - \frac{8\pi Ny}{|b|^3} \xi_0 - \frac{4 T^2}{b^2} = 0 $ and
   $b=\pm 2\frac{\sqrt{N}}{r}\cosh(v)$  we get
   \[
     -\frac{8T^2}{b\xi_0} =  -2b \xi_0  \pm \frac{16 \pi N n}{r^3 b^2} .
   \]
   Hence
  \begin{align*}
    h_2(\xi_0) & = -b \xi_0 - 2T \log \frac{1 + 2 T/(b \xi_0)}{1 - 2 T/( b \xi_0)} \\
    & =-b \xi_0 - \frac{8T^2}{b\xi_0} - 2T \sum_{2\leq j\leq J} \frac{2}{2j-1} \Big(\frac{2T}{b\xi_0}\Big)^{2j-1} + O\left( \frac{T^{1+2J\varepsilon}}{M^{2J-1}}\right) \\
    & = \mp 6\frac{\sqrt{N}}{r}\cosh(v) \xi_0 \pm  \frac{4 \pi  n}{r  \cosh^2(v)}
    - 2T \sum_{2\leq j\leq J} \frac{2}{2j-1} \Big(\frac{2T}{b\xi_0}\Big)^{2j-1} + O\left( \frac{T^{1+2J\varepsilon}}{M^{2J-1}}\right).
  \end{align*}
  Note that
   \begin{align*}
    \mp 6\frac{\sqrt{N}}{r}\cosh(v)  \xi_0 & \pm  \frac{4 \pi  n}{r  \cosh^2(v)}
    = \frac{ \pm  \pi  n}{ r\cosh^2(v)}
     \mp  \frac{ 3 \pi  n}{ r\cosh^2(v)}
    \sqrt{1 +   \frac{4 T^2 r^2 \cosh^4(v)}{ \pi^2  n^2}  } \\
    & = \frac{\mp 2 \pi  n}{ r\cosh^2(v)}
     \mp \frac{6  T^2 r  \cosh^2(v)}{\pi  n} \\
     & \hskip 2cm
     \mp   \frac{  n}{ r \cosh^2(v)} \sum_{2\leq j\leq J} c_j \left( \frac{Tr\cosh^2(v)}{n} \right)^{2j}
    + O\left( \frac{\sqrt{N}}{r} \frac{T^{2J\varepsilon}}{M^{2J}} \right) \\
    & =\frac{\mp 2 \pi  n}{ r }
      \mp \frac{6  T^2 r  }{\pi  n}
     \pm \frac{6  \pi  n}{ r } v^2
     +\frac{n}{r} \sum_{2\leq j\leq J} c_j' v^{2j}
     +  \frac{T^2 r }{ n} \sum_{1\leq j\leq J} c_j'' v^{2j}
     \\
    & \hskip 2cm  \mp   \frac{  n}{ r \cosh^2(v)} \sum_{2\leq j\leq J} c_j \left( \frac{Tr\cosh^2(v)}{n} \right)^{2j}
    + O\left( \frac{\sqrt{N}}{r} \frac{T^{2J\varepsilon}}{M^{2J}} \right) .
   \end{align*}
  By making a change of variable $v=\xi/M$, we get
  \begin{multline*}
    I_{\sigma_1}^\pm(n,r) = \frac{N^{1/2}}{M} \Big(\frac{n}{r^3}\Big)^{-1/2}
    e\left( \frac{\mp   n}{ r }
      \mp \frac{3  T^2 r  }{\pi^2  n}
      \right)
      \\ \cdot
    \int_{|\xi|\leq M^\varepsilon } g(\xi) V_3\left( \xi ,\frac{n}{\sqrt{N}}\right)
    e\left(\pm \frac{3  n}{ r M^2} \xi^2 +\frac{T}{\pi M} \xi\right) \dd \xi + O(T^{-A}),
  \end{multline*}
  where $V_3$ is a certain $N^{1/2+\varepsilon}/(rM^{4})$ inert function.
  Note that $n\asymp \sqrt{N}$ and $r\ll \frac{\sqrt{N}}{T^{1-\varepsilon}M}$. We have $\frac{n}{ r M^2} \gg \frac{T^{1-\varepsilon}}{M}$.
  We  apply a smooth partition of unity to get
  \begin{equation*}
    I_{\sigma_1}^\pm(n,r) = \sum_{\substack{M^{-1+\varepsilon} \leq \Xi \leq T^\varepsilon \\ \rm dyadic}}
    I_{\sigma_1}^\pm(n,r;\Xi) + O\left( r^{3/2} M^{-2+\varepsilon} \right),
  \end{equation*}
  where
  \begin{align*}
    I_{\sigma_1}^\pm(n,r;\Xi)
    & = \frac{N^{1/2}}{M} \Big(\frac{n}{r^3}\Big)^{-1/2}
    e\left( \frac{\mp   n}{ r }
      \mp \frac{3  T^2 r  }{\pi^2  n}
      \right)
      \sum_{\sigma_2=\pm1} \sigma_2
      \\
    & \hskip 1cm \cdot
    \int_{\mathbb{R}} V\left( \frac{\xi }{\Xi}\right) g(\sigma_2 \xi) V_3\left(\sigma_2 \xi ,\frac{n}{\sqrt{N}}\right)
    e\left(\pm \frac{3  n}{ r M^2} \xi^2 +\sigma_2 \frac{T}{\pi M} \xi\right) \dd \xi \\
    & = \frac{N^{1/2}}{M} \Big(\frac{n}{r^3}\Big)^{-1/2}
    e\left( \frac{\mp   n}{ r }
      \mp \frac{3  T^2 r  }{\pi^2  n}
      \right) \sum_{\sigma_2=\pm1} \sigma_2 \Xi
      \\
    &  \hskip 1cm \cdot
    \int_{\mathbb{R}} V\left(  \xi  \right) g(\sigma_2 \Xi \xi) V_3\left(\sigma_2 \Xi \xi ,\frac{n}{\sqrt{N}}\right)
    e\left(\pm \frac{3  n}{ r M^2}\Xi^2 \xi^2 +\sigma_2 \frac{T}{\pi M}\Xi \xi\right) \dd \xi.
  \end{align*}

  Let
  \[
    h_3(\xi) = \pm \frac{3  n}{ r M^2}\Xi^2 \xi^2 +\sigma_2 \frac{T}{\pi M}\Xi \xi.
  \]
  Then we have
  \[
    h_3'(\xi) = \pm \frac{6  n}{ r M^2}\Xi^2  \xi  + \sigma_2\frac{T}{\pi M}\Xi,
  \]
  and
  \[
    h_3''(\xi) = \pm \frac{6  n}{ r M^2} \Xi^2  , \quad
    h_3^{(j)}(\xi)=0, \quad j\geq3.
  \]

  If $\sgn \sigma_2=\pm$ or $\Xi \geq   \frac{2023 rTM}{\sqrt{N}}$  or $\Xi \leq  \frac{rTM}{2023\sqrt{N}}$, then $h_3'(\xi) \gg \frac{\sqrt{N}}{rM^2}+\frac{T}{M}$.
  By Lemma \ref{lemma:repeated_integration_by_parts}
  with $X=1$, $V^{-1}=1+\frac{N^{1/2+\varepsilon}}{rM^4}$, $R=\frac{\sqrt{N}}{rM^2}\Xi^2+\frac{T}{M}\Xi$, $Y/Q^2=\frac{\sqrt{N}}{rM^2}\Xi^2$,
  the contribution to
  $I_{\sigma_1}^\pm(n,r)$ is $ O(T^{-A})$.
  Here we have used the assumption $\Xi\geq M^{-1+\varepsilon}$ to verify $RV\gg T^\varepsilon$.

  If $\sgn \sigma_2=\mp$ and  $\Xi \asymp  \frac{rTM}{\sqrt{N}}$, then the solution of $h_3'(\xi)=0$ is $\xi_0= \frac{rTM}{6\pi n \Xi}$. Note that $h_3(\xi_0)=\mp \frac{rT^2}{12 \pi^2 n}$. Hence by Lemma \ref{lemma:stationary_phase}, we have
  \[
    I_{\sigma_1}^\pm(n,r) = r^2
    e\left( \frac{\mp   n}{ r }
      \mp \frac{37  T^2 r  }{12 \pi^2  n}
      \right)
      V_r\left(  \frac{n}{\sqrt{N}}\right) + O\left( r^{3/2} M^{-2+\varepsilon} \right),
  \]
  as claimed.
\end{proof}

\subsection{Applying the functional equation}

By \eqref{eqn:RS3}, the contribution from the error terms in  Lemma \ref{lemma:I} to \eqref{eqn:R+<<I} can be bounded by
\[
  \ll \frac{T^\varepsilon}{M^2} \frac{TM}{\sqrt{N}} \sum_{1\leq r\ll \frac{\sqrt{N}}{T^{1-\varepsilon}M} } \frac{1}{r^{1/2}}
   \sum_{n\asymp \sqrt{N}}  |A(1,n)|  \ll T,
\]
provided $M\geq T^{1/3}$ and $N\leq T^{3+\varepsilon}$.  Hence we have
\begin{align}\label{eqn:R+<<sum}
  \mathcal{R}^{+}(N)
   &  \ll \bigg| \frac{MT}{N^{1/2}}
  \sum_{1\leq r\ll \frac{\sqrt{N}}{T^{1-\varepsilon}M} }
   \sum_{n\geq1}  A(1,n)
    e\left( \mp \frac{37  T^2 r  }{12 \pi^2  n}
      \right)
      V_r\left(\frac{n}{\sqrt{N}}\right) \bigg|
   + O(T) .
\end{align}


Recall that we have the following functional equation
\[
  L_\infty(s,\Sym^2\phi) L (s,\Sym^2\phi) = L_\infty(1-s,\Sym^2\phi) L(1-s,\Sym^2\phi).
\]
As a consequence of the functional equation, we have the following summation formula, which is the same as the usual $\GL(3)$ Voronoi summation formula with the trivial additive character.

\begin{lemma}\label{lemma:FE3}
  For $w\in C_c^\infty((0,\infty))$, we have
  \[
    \sum_{n=1}^{\infty} A(1,n) w(n) = \sum_{n=1}^{\infty} \frac{A(n,1)}{n} W(n),
  \]
  where
  \begin{equation}\label{eqn:W1}
    W(n) = \frac{1}{2\pi i} \int_{(\sigma)} \tilde{w}(s) \frac{L_\infty(1-s,\Sym^2 \phi)}{L_\infty(s,\Sym^2 \phi)} n^s  \dd s,
  \end{equation}
  for any $\sigma<3/4$.
\end{lemma}

\begin{proof}
  This follows from the functional equation.
\end{proof}

By Lemma \ref{lemma:FE3} with $w_r(n) = e\left( \mp \frac{37  T^2 r  }{12 \pi^2  n}
      \right)   V_r\left( \frac{n}{\sqrt{N}}\right) $, we get
\begin{align}\label{eqn:R+<<sumrn}
  \mathcal{R}^{+}(N)
   &  \ll \bigg|  \frac{MT}{N^{1/2}}
  \sum_{1\leq r\ll \frac{\sqrt{N}}{T^{1-\varepsilon}M} }
  \sum_{n\geq1}  \frac{A(1,n)}{n}
   W_r(n) \bigg| + O(T^{-A}),
\end{align}
where $W_r(n)$ is defined as in \eqref{eqn:W1} with $w_r$.
We will use the method in the proof of Lemma \ref{lemma:W} to deal with $W_r(n)$.

\begin{lemma}\label{lemma:W1}
  Assume $1\leq r\ll \frac{\sqrt{N}}{T^{1-\varepsilon}M} $ and $T^{1/3}\leq M\leq T^{1/2}$. We have  $W_r(n) \ll n^{-5 }T^{-2022}$ unless $n\asymp \frac{T^2 r^{1/2}}{N^{1/2}}$, in which case we have
  \[
    W_r(n) \ll \frac{T^{3/2+\varepsilon}}{M^{1/2}}.
  \]
\end{lemma}

\begin{proof}
  If $n\leq T^{100}$, then we have
  \[
    W_r(n) = \frac{1}{2\pi} \int_{\mathbb{R}} \tilde{w}_r(1/2+i\tau) \frac{L_\infty(1/2-i\tau,\Sym^2 \phi)}{L_\infty(1/2+i\tau,\Sym^2 \phi)} n^{1/2+i\tau}  \dd \tau.
  \]
  We first deal with $\tilde{w}_r(1/2+i\tau) $. Note that
  \begin{align*}
    \tilde{w}_r(1/2+i\tau)
    & = \int_{0}^{\infty}
    e\left( \mp \frac{37  T^2 r  }{12 \pi^2  y}
      \right)   V_r\left(\frac{y}{\sqrt{N}}\right) y^{-1/2+i\tau} \dd y \\
    & = N^{1/4+i\tau/2} \int_{0}^{\infty}
    e\left( \mp \frac{37  T^2 r  }{12 \pi^2  \sqrt{N}}\xi - \frac{\tau}{2\pi} \log \xi
      \right)   V_r\left(\frac{1}{\xi} \right) \xi^{-3/2} \dd \xi.
  \end{align*}
  Let
  \[
    h_4(\xi) = \mp \frac{37  T^2 r  }{12 \pi^2  \sqrt{N}}\xi - \frac{\tau}{2\pi} \log \xi.
  \]
  Then
  \[
    h_4'(\xi) = \mp \frac{37  T^2 r  }{12 \pi^2  \sqrt{N}}  - \frac{\tau}{2\pi  \xi}
  \]
  and
  \[
    h_4''(\xi) =  \frac{\tau}{2\pi \xi^2}  , \quad
    h_4^{(j)}(\xi) \asymp |\tau|, \quad j\geq2.
  \]
  Note that $T^{1/2-\varepsilon} \ll T^2 r/ \sqrt{N} \ll T^{1+\varepsilon}/M$, $T^{1/3}\leq M\leq T^{1/2}$, and $N\leq T^{3+\varepsilon}$. Hence
  \[
    \frac{T^2 r}{\sqrt{N}} \gg T^{1/8} \left( 1+ \frac{N^{1/2+\varepsilon}}{rM^4} \right)^2.
  \]
  Recall that $V_r$ is a $(1+ \frac{N^{1/2+\varepsilon}}{rM^4})$ inert function.
  By Lemma \ref{lemma:repeated_integration_by_parts},
  we get $\tilde{w}_r(1/2+i\tau) \ll (1+|\tau|)^{-5} T^{-A}$ unless  $\sgn(\tau)=\mp$ and $\mp \tau \asymp T^2 r/ \sqrt{N}$.

  Assume  $\sgn(\tau)=\mp$ and $\mp \tau \asymp T^2 r/ \sqrt{N}$. The solution of $h_4'(\xi)=0$ is
  $
    \xi_0 = \mp  \frac{6\pi \sqrt{N} \tau}{37 T^2 r}.
  $ And we have
  $
    h_4(\xi_0) = \frac{\tau}{2\pi} \log \frac{e 37 T^2 r}{6\pi \sqrt{N} |\tau|}.
  $
  So by Lemma \ref{lemma:stationary_phase} we have
  \begin{align*}
    \tilde{w}_r(1/2+i\tau)
    &  = N^{1/4+i\tau/2} \frac{1}{\sqrt{|\tau|}}
    e\left( \frac{\tau}{2\pi} \log \frac{e 37 T^2 r}{6\pi \sqrt{N} |\tau|}
      \right)   V_{r,2}\left( \mp  \frac{6\pi \sqrt{N} \tau}{37 T^2 r}  \right)
      + O(T^{-A}) ,
  \end{align*}
  where $V_{r,2}$ is certain $(1+ \frac{N^{1/2+\varepsilon}}{rM^4})$ inert function.

  By definition we have
  \[
    \frac{L_\infty(1/2-i\tau,\Sym^2 \phi)}{L_\infty(1/2+i\tau,\Sym^2 \phi)}
    = \pi^{3i\tau} \frac{\Gamma(\frac{1/2-i\tau-2iT}{2})\Gamma(\frac{1/2-i\tau+2iT}{2}) \Gamma(\frac{1/2-i\tau}{2})} {\Gamma(\frac{1/2+i\tau-2iT}{2})\Gamma(\frac{1/2+i\tau+2iT}{2}) \Gamma(\frac{1/2+i\tau}{2})}.
  \]
  By the Stirling's formula we get
  \begin{multline*}
    \frac{L_\infty(1/2-i\tau,\Sym^2 \phi)}{L_\infty(1/2+i\tau,\Sym^2 \phi)}
    =  \pi^{3i\tau}  \exp\Big( i(-\tau-2T)\log \frac{2T+\tau}{2e}  \\
    + i(-\tau+2T)\log \frac{2T-\tau}{2e} + i(-\tau)\log \frac{\mp\tau}{2e} \Big)
    (w_J(\tau) + O(|\tau|^{-J})).
  \end{multline*}
  Hence we have
  \begin{multline*}
    W_r(n) = \frac{ N^{1/4}n^{1/2}}{2\pi} \int_{\mathbb{R}}
    V_{r,2}\left( \mp  \frac{6\pi \sqrt{N} \tau}{37 T^2 r}  \right) \frac{w_J(\tau)}{\sqrt{|\tau|}}
    e\left( \frac{\tau}{2\pi} \log \frac{e 37 T^2 r}{6\pi \sqrt{N} |\tau|}
      \right)
        \exp\left( i \tau \log \pi^3 \sqrt{N} n\right) \\
         \exp\Big( i(-\tau-2T)\log \frac{2T+\tau}{2e}
    + i(-\tau+2T)\log \frac{2T-\tau}{2e} + i(-\tau)\log \frac{\mp\tau}{2e} \Big)
    \dd \tau + O(T^{-A}).
  \end{multline*}
  Making a change of variable $\mp \tau=Y\xi$ where $Y=\frac{37 T^2 r}{6\pi \sqrt{N}}$, we get
  \begin{multline*}
    W_r(n) =  Y^{1/2} N^{1/4}n^{1/2}
    \int_{\mathbb{R}}  V\left( \xi  \right)
        \exp\left( \mp  i Y \xi \log \frac{e  }{\xi} \mp  i Y \xi  \log \pi^3 \sqrt{N} n\right) \\
         \exp\Big( i(\pm Y\xi-2T)\log \frac{2T\mp Y \xi }{2e}
    + i(\pm Y \xi +2T)\log \frac{2T\pm Y \xi }{2e} \pm i Y \xi  \log \frac{Y \xi }{2e} \Big)
    \dd \xi + O(T^{-A}),
  \end{multline*}
  where $V$ is a $1+ \frac{N^{1/2+\varepsilon}}{rM^4}$ inert function.
  Let
  \begin{multline*}
    h_5(\xi) =   \mp    Y \xi \log \frac{e}{\xi}
    \mp  Y \xi  \log \pi^3 \sqrt{N} n \\
     +  (\pm Y\xi-2T)\log \frac{2T\mp Y \xi }{2e}
    + (\pm Y \xi +2T)\log \frac{2T\pm Y \xi }{2e} \pm  Y \xi  \log \frac{Y \xi }{2e}
    \\
    = \pm  Y \xi \log \frac{\xi}{e \pi^3 \sqrt{N} n}  \frac{2T\mp Y \xi }{2e} \frac{2T\pm Y \xi }{2e} \frac{Y \xi }{2e}
    +  2T \log \frac{2T\pm Y \xi }{2T\mp Y \xi }  .
  \end{multline*}
  We have
  \begin{multline*}
    h_5'(\xi) =   \mp    Y  \log \frac{e}{\xi}  \pm   Y
    \mp  Y  \log \pi^3 \sqrt{N} n
      \pm Y \log \frac{2T\mp Y \xi }{2}
     \pm Y \log \frac{2T\pm Y \xi }{2}
    \pm  Y   \log \frac{Y \xi }{2}
    \\
    = \pm   Y  \log  \frac{ Y\xi^2 (4T^2-Y^2\xi^2)}{8\pi^3 e^2  \sqrt{N} n} ,
  \end{multline*}
  \[
    h_5''(\xi) = \pm 2 Y \frac{1}{\xi} \pm Y \frac{\mp Y}{2T\mp Y\xi} \pm Y \frac{\pm Y}{2T\pm Y\xi},
  \]
  \[
    h_5^{(j)}(\xi) \asymp Y \asymp \frac{T^2 r}{\sqrt{N}}, \quad j\geq2.
  \]

  If $n\leq \frac{T^2 Y}{2023 \cdot \pi^3 e^2 \sqrt{N}}$ or $n \geq \frac{2023  T^2 Y}{ \pi^3 e^2 \sqrt{N}}$, then $h_5'(\xi)\gg Y$. Hence by Lemma \ref{lemma:repeated_integration_by_parts}, we get
  $W_r(n) \ll T^{-A}$ for any $A>0$.

  If $n\asymp \frac{T^2 Y}{\sqrt{N}}$, then by the second derivative test (e.g.\ Lemma \ref{lemma:stationary_phase}) we get
  \[
    W_r(n) \ll  N^{1/4}n^{1/2} \ll T  \frac{Tr^{1/2}}{N^{1/4}}  \ll \frac{T^{3/2+\varepsilon}}{M^{1/2}}.
  \]
  Here we have used $r\ll \frac{\sqrt{N}}{T^{1-\varepsilon}M}$.

  If $n\geq T^{100}$, then we move the line of integration to $\Re(s)=-a$ for $a>0$, getting
  \[
    W_r(n) = \frac{1}{2\pi} \int_{\mathbb{R}} \tilde{w}_r(-a+i\tau) \frac{L_\infty(1+a-i\tau,\Sym^2 \phi)}{L_\infty(-a+i\tau,\Sym^2 \phi)} n^{-a+i\tau}  \dd \tau.
  \]
  By the same argument as for $\tilde{w}_r(1/2+i\tau)$, we have  $\tilde{w}_r(-a+i\tau) \ll (1+|\tau|)^{-5} T^{-A}$ unless  $\sgn(\tau)=\mp$ and $\mp \tau \asymp T^2 r/ \sqrt{N}$, in which case we have
  \begin{align*}
    \tilde{w}_r(-a+i\tau)
    &  \ll N^{-a/2} .
  \end{align*}
  By the Stirling's formula we get
  \begin{equation*}
    \frac{L_\infty(1+a-i\tau,\Sym^2 \phi)}{L_\infty(-a+i\tau,\Sym^2 \phi)}
    \ll T^{1+2a} |\tau|^{1/2+a}.
  \end{equation*}
  Hence we have
  \begin{equation*}
    W_r(n) \ll
    \int_{|\tau| \asymp \frac{T^2r}{\sqrt{N}}}  N^{-a/2} T^{1+2a} |\tau|^{1/2+a} n^{-a}  |\dd \tau|
     + O(n^{-5}T^{-A}) \ll n^{-5} T^{-2022},
  \end{equation*}
  by taking $a$ large enough.
  This completes the proof.
\end{proof}

By \eqref{eqn:R+<<sumrn}, Lemma \ref{lemma:W1} and the Rankin--Selberg bound \eqref{eqn:RS3}, we have
\begin{align}\label{eqn:R+<<}
  \mathcal{R}^{+}(N)
   &  \ll \bigg|  \frac{MT}{N^{1/2}}
  \sum_{1\leq r\ll \frac{\sqrt{N}}{T^{1-\varepsilon}M} }
  \sum_{n\asymp \frac{T^2 r^{1/2}}{N^{1/2}}}  \frac{|A(1,n)|}{n}
   \frac{T^{3/2}}{M^{1/2}} \bigg| + O(T)
  \ll \frac{T^{3/2+\varepsilon}}{M^{1/2}}.
\end{align}

\subsection{The terms with $K$-Bessel function}

Now we consider $\mathcal{R}^-(N)$.
By Lemma \ref{lemma: H-} and making a change of variable $cm=r$, we get
\begin{multline*}
  \mathcal{R}^{-}(N)
  \ll \frac{MT}{N^{1/2}}  \sum_{r\asymp \frac{\sqrt{N}}{T} }\frac{1}{r}
  \bigg| \int_{|v|\leq \frac{M^\varepsilon}{M}}  \sum_{m\mid r} \sum_{n\geq1} m A(m,n)  S(n, 1;r/m) V\left(\frac{m^2 n}{N} \right)  \\
  \cdot \sum_{\sigma_1=\pm }
  e\left(\sigma_1\frac{2m\sqrt{n}}{r}\sinh(v)\right)
  e\left(\frac{vT}{\pi}\right)g(Mv)\dd v \bigg| + O(T^{-A}) .
\end{multline*}
Note that $\sinh(v)=v+O(|v|^3)$ if $v=o(1)$.
By repeated integration by parts, the contribution from $\sigma_1=+$ is $O(T^{-A})$ for any $A>0$.
By Lemma \ref{lemma:VSF}, we get
\begin{align} \label{eqn:R-<<sum}
  \mathcal{R}^{-}(N)
   & \ll  \frac{MT}{N^{1/2}} \sum_{r\asymp \frac{\sqrt{N}}{T} }\frac{1}{r}
    \bigg| \int_{|v|\leq \frac{M^\varepsilon}{M}}
  \sum_{\pm } \sum_{n\geq1} \frac{A(1,n)}{r} e\left(\pm \frac{n}{r}\right)
  W_{v}^\pm \left( \frac{n}{r^3} \right)
  e\left(\frac{vT}{\pi}\right)g(Mv)\dd v \bigg| \nonumber \\
  & \hskip 10cm + O(T^{-A}) \nonumber \\
  &  \ll  \frac{MT}{N^{1/2}} \sum_{r\asymp \frac{\sqrt{N}}{T} }\frac{1}{r^2} \sum_{\pm }
    \bigg|
  \sum_{n\geq1}  A(1,n)  e\left(\pm \frac{n}{r}\right)  I^\pm\left(\frac{n}{r^3}\right)  \bigg|
   + O(T^{-A}),
\end{align}
where $ W_{v}^\pm$ is given by \eqref{eqn:W+-} with $w_{v}(y)=  V\left(\frac{y}{N} \right)
  e\left(-\frac{2\sqrt{y}}{r}\sinh(v)\right) $, and
\[
  I^\pm(x) := \int_{|v|\leq \frac{M^\varepsilon}{M}} W_{v}^\pm \left( x \right)
  e\left(\frac{vT}{\pi}\right)g(Mv)\dd v .
\]

\begin{lemma}\label{lemma:I-}
  Assume $T^{1/3+\varepsilon}\leq M\leq T^{1/2-\varepsilon}$.
  If $x\geq T^{-10}$, then
  \[
    I^\pm (x) \ll \Big( \frac{T^3}{MN x} \Big)^{A} T^{100},
  \]
  for any $A>0$. If $x\leq T^{100}$, then we have
  \[
    I^\pm(x) = I_0^\pm(x) +  \sum_{\substack{T^\varepsilon\ll \Upsilon \ll T^{1+\varepsilon}/M \\ \rm dyadic}} I_{\Upsilon}^\pm(x)  + O(T^{-A})
  \]
  where
  \begin{align*}  
    I_{\Upsilon}^\pm(x)  &
    = \frac{N^{1/2} x^{-1/2} \Upsilon^{1/2}}{2\pi M}
    \int_{0}^{\infty}  V\left(\xi \right)\xi^{-1/2}
    \hat{k}\left(\frac{2\sqrt{N\xi}}{r M} -\frac{T}{\pi M} \right)  \\
    & \hskip 3cm \cdot
    V(\eta_0 + \eta_1 +\eta_2 ) e^{ih_\xi(x)} \dd \xi
    + O\left( \frac{N^{1/2} x^{-1/2}  \Upsilon^{15/2}}{T^{7}}\right),
  \end{align*}
  with $\eta_j$ in \eqref{eqn:eta} and $h_\xi(x)$ in \eqref{eqn:h_xi}, and
  \begin{align*}
    I_0^\pm (x) &
    \ll \frac{N^{1/2} x^{-1/2}}{T} T^\varepsilon.
  \end{align*}
\end{lemma}

\begin{proof}
  For any $v\ll M^{\varepsilon-1}$, we have
  \begin{align}\label{eqn:w_v<<}
    \tilde{w}_v(s)
    & = \int_{0}^{\infty}  V\left(\frac{y}{N} \right)
      e\left(-\frac{2\sqrt{y}}{r}\sinh(v)\right) y^{s-1} \dd y \nonumber \\
    & = N^{s} \int_{0}^{\infty}  V\left(\xi \right)
      e\left(-\frac{2\sqrt{N\xi}}{r}\sinh(v)\right) \xi^{s-1} \dd \xi
      \ll N^{\Re(s)}\left(\frac{T^{1+\varepsilon}/M}{1+|s|}\right)^{k},
  \end{align}
  for any $k\in\mathbb{Z}_{\geq0}$.
  By Stirling's formula we have
  \[
    G^\pm(\sigma+i\tau) \ll_{\sigma} (1+|\tau+2T|)^{1/2-\sigma}(1+|\tau-2T|)^{1/2-\sigma}(1+|\tau|)^{1/2-\sigma}
  \]
  Hence by shifting the contour to $\Re(s)=-A$, we get
  \begin{align*}
    W_v^\pm (x)
    & \ll \int_{\mathbb{R}}  \Big((1+|\tau+2T|) (1+|\tau-2T|) (1+|\tau|)\Big)^{1/2+A} N^{-A}\left(\frac{T^{1+\varepsilon}/M}{1+|\tau|}\right)^{k} x^{-A-1} \dd \tau  \\
    & \ll \int_{|\tau|\ll T/M} + \int_{T/M\leq |\tau|\leq T} + \int_{|v|\geq T} \\
    & \ll \Big(\frac{T^3}{M}\Big)^{A+1} N^{-A} x^{-A-1}
    + T^{2A+2} \Big(\frac{T}{M}\Big)^{A+3} N^{-A} x^{-A-1}
    + \Big(\frac{T}{M}\Big)^{3A+5} N^{-A} x^{-A-1} .
  \end{align*}
  For  $x\gg T^{-10}$, we get
  \[
    I^\pm (x) \ll \Big( \frac{T^3}{MN x} \Big)^{A} T^{100},
  \]
  for any $A>0$.

  If $x\leq T^{100}$, then by shifting the contour to $\Re(s)=1/2$, we get
  \begin{align*}
    I^\pm (x) &
    = \int_{|v|\leq \frac{M^\varepsilon}{M}} \frac{1}{2\pi} \int_{|\tau|\leq \frac{T^{1+\varepsilon}}{M}} G^\pm(1/2+i\tau) \tilde{w}_v(1/2+i\tau) x^{-1/2+i\tau} \dd \tau e\left(\frac{vT}{\pi}\right)g(Mv)\dd v \\
    & \hskip 8cm
    + O(T^{-A}).
  \end{align*}
  By \eqref{eqn:w_v<<} and a smooth partition of unity for $\tau$-integral,  we get
  \begin{align*}
    I^\pm (x)
    & = I_0^\pm(x) + \sum_{\sigma_3=\pm}\sum_{\substack{T^\varepsilon\ll \Upsilon \ll T^{1+\varepsilon}/M \\ \rm dyadic}} I_{\sigma_3,\Upsilon}^\pm(x)  + O(T^{-A}),
  \end{align*}
  where
  \begin{align*}
    I_{\sigma_3,\Upsilon}^\pm (x) &  =
  \frac{N^{1/2} x^{-1/2}}{2\pi}
    \int_{0}^{\infty}  V\left(\xi \right)
    \int_{\mathbb{R}} G^\pm(1/2+i\tau) V\left( \frac{\sigma_3\tau}{\Upsilon} \right)
    N^{i\tau} x^{i\tau} \xi^{-1/2+i\tau} \\
    & \hskip 1cm \cdot \int_{|v|\leq \frac{M^\varepsilon}{M}}
      e\left(-\frac{2\sqrt{N\xi}}{r}\sinh(v)\right) e\left(\frac{vT}{\pi}\right)g(Mv)\dd v  \dd \tau \dd \xi
  \end{align*}
  and
  \begin{align*}
    I_{0}^\pm (x) &  =
  \frac{N^{1/2} x^{-1/2}}{2\pi}
    \int_{0}^{\infty}  V\left(\xi \right)
    \int_{\mathbb{R}} G^\pm(1/2+i\tau) U\left( \frac{\tau}{T^{\varepsilon}} \right)
    N^{i\tau} x^{i\tau} \xi^{-1/2+i\tau} \\
    & \hskip 1cm \cdot \int_{|v|\leq \frac{M^\varepsilon}{M}}
      e\left(-\frac{2\sqrt{N\xi}}{r}\sinh(v)\right) e\left(\frac{vT}{\pi}\right)g(Mv)\dd v  \dd \tau \dd \xi
  \end{align*}
  for some $U$ being a smooth function such that $U(u)=1$ if $u\in[-1,1]$ and $U(u)=0$ if $|u|\geq2$ and $U^{(j)}(u)\ll_j 1$ for all $u\in\mathbb{R}$.

  By the Taylor expansion of $\sinh$ we get
  \begin{align*}
    I_{\sigma_3,\Upsilon}^\pm (x) &
    = \frac{N^{1/2} x^{-1/2}}{2\pi}
    \int_{0}^{\infty}  V\left(\xi \right)
    \int_{\mathbb{R}} G^\pm(1/2+i\tau) V\left( \frac{\sigma_3\tau}{\Upsilon} \right)
    N^{i\tau} x^{i\tau} \xi^{-1/2+i\tau} \\
    & \hskip 0cm \cdot \int_{|v|\leq \frac{M^\varepsilon}{M}}
      e\left(\frac{vT}{\pi}-\frac{2\sqrt{N\xi}}{r}v \right) e\left(-\frac{2\sqrt{N\xi}}{r}\sum_{1\leq j\leq J} c_j v^{2j+1}\right)g(Mv)\dd v  \dd \tau \dd \xi  + O(T^{-2022}),
  \end{align*}
  for some large enough $J$.
  Since $g^{(j)}(y)\ll_{j,A} (1+|y|)^{-A}$, we can extend the $v$ integral to $\mathbb{R}$ with a negligibly small error. By making a change of variable $Mv \rightarrow v$, we get
  \begin{align*}
    I_{\sigma_3,\Upsilon}^\pm (x) &
    = \frac{N^{1/2} x^{-1/2}}{2\pi M}
    \int_{0}^{\infty}  V\left(\xi \right)
    \int_{\mathbb{R}} G^\pm(1/2+i\tau)  V\left( \frac{\sigma_3\tau}{\Upsilon} \right)
    N^{i\tau} x^{i\tau} \xi^{-1/2+i\tau} \\
    & \hskip 0.5cm \cdot \int_{\mathbb{R}}
      e\left(\frac{vT}{\pi M}-\frac{2\sqrt{N\xi}}{r M}v \right) e\left(-\frac{2\sqrt{N\xi}}{r}\sum_{1\leq j\leq J} c_j \frac{v^{2j+1}}{M^{2j+1}}\right)g(v)\dd v  \dd \tau \dd \xi  + O(T^{-2022}).
  \end{align*}
  Let
  \begin{equation}\label{eqn:k=}
    k(v)=e\left(-\frac{2\sqrt{N\xi}}{r}\sum_{1\leq j\leq J} c_j \frac{v^{2j+1}}{M^{2j+1}}\right)g(v) .
  \end{equation}
  Since $M\geq T^{1/3+\varepsilon}$ and $\frac{\sqrt{N}}{r} \asymp T$, we know that $k$ is a Schwartz function. So
  \begin{align}\label{eqn:I_U}
    I_{\sigma_3,\Upsilon}^\pm (x) &
    = \frac{N^{1/2} x^{-1/2}}{2\pi M}
    \int_{0}^{\infty}  V\left(\xi \right)\xi^{-1/2}  \hat{k}\left(\frac{2\sqrt{N\xi}}{r M} -\frac{T}{\pi M} \right)  \\
    & \hskip 3cm \cdot
    \int_{\mathbb{R}} G^\pm(1/2+i\tau) V\left( \frac{\sigma_3\tau}{\Upsilon} \right)
    (N x \xi)^{i\tau}  \dd \tau \dd \xi  + O(T^{-2022}) \nonumber.
  \end{align}
  By \eqref{eqn:G==}, if $\sigma_3=\pm$ then we have the $\tau$-integral
  \[
    J(x) = \int_{\mathbb{R}} G^\pm(1/2+i\tau) V\left( \frac{\sigma_3\tau}{\Upsilon} \right)
    (N \xi x)^{i\tau}  \dd \tau \ll T^{-A};
  \]
  and if $\sigma_3=\mp$ then we have
  \begin{multline*}
    J(x) = \int_{\mathbb{R}} \exp\bigg( -i(\tau+2T)\log \frac{\tau+2T}{e} -i(\tau-2T)\log \frac{2T-\tau}{e} \\
    - i\tau \log \frac{\mp \tau}{(2\pi)^{3} e} \bigg)
    w_{J}(\tau)
    V\left( \frac{\mp\tau}{\Upsilon} \right)
    (N \xi x)^{i\tau}  \dd \tau + O(T^{-A}).
  \end{multline*}
  Making a change of variable $\tau=\mp \Upsilon \eta$, we get
  \begin{multline*}
    J(x) = \Upsilon \int_{\mathbb{R}} w_{J}(\mp \Upsilon \eta)
    V\left( \eta \right)
    \exp\bigg( -i(2T \mp \Upsilon \eta)\log \frac{2T\mp \Upsilon \eta}{e} \\
    -i(\mp \Upsilon \eta-2T)\log \frac{2T\pm \Upsilon \eta}{e}
    - i (\mp \Upsilon \eta) \log \frac{ \Upsilon \eta}{(2\pi)^{3} e}
     \mp i \Upsilon \eta \log (N \xi x) \bigg) \dd \eta + O(T^{-A}).
  \end{multline*}
  Let
  \begin{multline*}
    h_6(\eta) = -(2T \mp \Upsilon \eta)\log \frac{2T\mp \Upsilon \eta}{e} \\
    -(\mp \Upsilon \eta-2T)\log \frac{2T\pm \Upsilon \eta}{e}
    -  (\mp \Upsilon \eta) \log \frac{ \Upsilon \eta}{(2\pi)^{3} e}
     \mp  \Upsilon \eta \log (N \xi x)
     \\
    = \pm \Upsilon \eta \log  \frac{\Upsilon \eta (4T^2 - \Upsilon^2 \eta^2)} {(2\pi)^{3} e^3 N \xi x}
    + 2T \log \frac{2T\pm \Upsilon\eta}{2T\mp \Upsilon\eta}.
  \end{multline*}
  Then we have
  \begin{multline*}
    h_6'(\eta) = \pm \Upsilon \log (2T\mp \Upsilon \eta )
    \pm \Upsilon  \log (2T\pm \Upsilon \eta)
    \pm \Upsilon  \log \frac{ \Upsilon \eta}{(2\pi)^{3} }
     \mp  \Upsilon\log (N \xi x)
     \\
    = \pm \Upsilon  \log  \frac{\Upsilon \eta (4T^2 - \Upsilon^2 \eta^2)}{(2\pi)^{3}   N \xi x}.
  \end{multline*}
  \begin{equation*}
    h_6''(\eta) = \pm \Upsilon \frac{\mp \Upsilon}{2T\mp \Upsilon \eta}
    \pm \Upsilon \frac{\pm \Upsilon}{2T\pm \Upsilon \eta}
    \pm \Upsilon \frac{1}{\eta} , \quad
    h^{(j)}(\eta) \asymp \Upsilon, \quad j\geq2.
  \end{equation*}
  So we have
  $J(x) \ll T^{-A}$ unless $Nx\asymp T^2 \Upsilon$. Hence $I_{\sigma_3,\Upsilon}^\pm (x) \ll T^{-A}$ unless $\sigma_3=\mp$ and $Nx\asymp T^2 \Upsilon$, in which case we have  the solution of $h_6'(\eta)=0$ is
  \[
    \eta = \eta_0 + \eta_1 + \eta_2+ \eta_3,
  \]
  where
  \begin{equation} \label{eqn:eta}
  \begin{split}
    \eta_0 & = \frac{2\pi^3 N\xi x}{T^2 \Upsilon}, \\
    \eta_1 & = \frac{8\pi^3 N\xi x}{4 T^2 \Upsilon} \frac{\Upsilon^3}{(2\pi)^3 N\xi x} \eta_0^3
    = \frac{\Upsilon^2}{4T^2} \eta_0^3, \\
    \eta_2 & =  \frac{\Upsilon^2}{4T^2} 3 \eta_0 \eta_1 = \frac{3 \Upsilon^4}{16T^4} \eta_0^4, \\
    \eta_3 &  \ll \frac{\Upsilon^6}{T^6}.
  \end{split}
  \end{equation}
  By a Taylor series expansion, we have
  \[
    h_6( \eta_0 + \eta_1 + \eta_2 + \eta_3) =
    h_\xi(x)
    + O\left(\frac{\Upsilon^7}{T^6}\right),
  \]
  where
  \begin{equation}\label{eqn:h_xi}
    h_\xi(x) = \mp   \Upsilon (\eta_0 + \eta_1 + \eta_2)
    \pm \frac{2}{3} \frac{\Upsilon^3 (\eta_0^3+3\eta_0\eta_1)}{4T^2}
    \pm \frac{2}{5} \frac{\Upsilon^5 \eta_0^5}{(2T)^4}.
  \end{equation}
  By the stationary phase method (Lemma \ref{lemma:stationary_phase}), we have
  \[
    J(x) = \Upsilon^{1/2} V(\eta_0 + \eta_1 +\eta_2 ) e^{i h_\xi(x)} \left(1+O\Big(\frac{\Upsilon^7}{T^6}\Big)\right) + O(T^{-A}).
  \]
  Hence we have
  \begin{align}  
    I_{\sigma_3,\Upsilon}^\pm (x) &
    = \frac{N^{1/2} x^{-1/2}}{2\pi M}
    \int_{0}^{\infty}  V\left(\xi \right)\xi^{-1/2}  \hat{k}\left(\frac{2\sqrt{N\xi}}{r M} -\frac{T}{\pi M} \right)  \\
    & \hskip 1cm \cdot
      \Upsilon^{1/2} V(\eta_0 + \eta_1 +\eta_2 ) e^{i h_\xi(x)} \left(1+O\Big(\frac{\Upsilon^7}{T^6}\Big)\right) \dd \xi  + O(T^{-2022}) \nonumber.
  \end{align}
  Since $\hat{k}$ is a Schwartz function, we know that the contribution from $\frac{2\sqrt{N\xi}}{r M} -\frac{T}{\pi M} \gg T^{\varepsilon}$ is negligible. Note that
  \[
    \frac{2\sqrt{N\xi}}{r M}  \xi^{1/2} -\frac{T}{\pi M}
    \ll T^{\varepsilon}
  \]
  is equivalent to
  \begin{equation}\label{eqn:xi}
    \xi - \frac{T^2 r^2 } {4\pi^2 N}  \ll \frac{T^\varepsilon M}{T}.
  \end{equation}
  Hence we get
  \begin{align}  
    I_{\sigma_3,\Upsilon}^\pm (x) &
    = \frac{N^{1/2} x^{-1/2} \Upsilon^{1/2}}{2\pi M}
    \int_{0}^{\infty}  V\left(\xi \right)\xi^{-1/2}
    \hat{k}\left(\frac{2\sqrt{N\xi}}{r M} -\frac{T}{\pi M} \right)  \\
    & \hskip 1cm \cdot
    V(\eta_0 + \eta_1 +\eta_2 ) e^{ih_\xi(x)} \dd \xi
    + O\left( \frac{N^{1/2} x^{-1/2}  \Upsilon^{15/2}}{T^{7}}\right) \nonumber.
  \end{align}

  By a similar argument as in \eqref{eqn:I_U}, we have
  \begin{align*}
    I_{0}^\pm  (x) &
    = \frac{N^{1/2} x^{-1/2}}{2\pi M}
    \int_{0}^{\infty}  V\left(\xi \right) \xi^{-1/2}
    \hat{k}\left(\frac{2\sqrt{N\xi}}{r M} -\frac{T}{\pi M} \right) \\
    & \hskip 3cm \cdot
    \int_{\mathbb{R}} G^\pm(1/2+i\tau) U\left( \frac{\tau}{T^{\varepsilon}} \right)
    (N  x  \xi)^{ i\tau} \dd \tau \dd \xi  + O(T^{-2022}).
  \end{align*}
  Hence by \eqref{eqn:xi} we get
  \begin{align*}
    I_0^\pm (x) &
    \ll \frac{N^{1/2} x^{-1/2}}{T} T^\varepsilon.
  \end{align*}
  This completes the proof.
\end{proof}

By  Lemma \ref{lemma:I-} we know that the contribution from the terms with
\[
  \frac{n}{r^3} \gg \frac{T^{3+\varepsilon}}{MN},
  \quad \textrm{that is,} \quad
  n \gg \frac{N^{1/2} T^{\varepsilon}}{M}
\]
is negligibly small.
By \eqref{eqn:R-<<sum} and Lemma \ref{lemma:I-}, we get the contribution from $I_0^\pm(x)$
 to  $\mathcal{R}^{-}(N)$ is
\begin{align*} \label{eqn:R0-<<sum}
  \mathcal{R}_0^{-}(N)
   &  \ll  \frac{MT}{N^{1/2}} \sum_{r\asymp \frac{\sqrt{N}}{T} }\frac{1}{r^2}
  \sum_{n\leq \frac{N^{1/2} T^{\varepsilon}}{M}}  |A(1,n)|  \frac{N^{1/2}}{T} \frac{r^{3/2}}{n^{1/2}}
   + O(T^{-A}).
\end{align*}
By the Rankin--Selberg bound \eqref{eqn:RS3}, we get
\begin{equation} \label{eqn:R-<<}
  \mathcal{R}_0^{-}(N)
    \ll \frac{N^{1/2+\varepsilon} M^{1/2}}{T^{1/2}} \ll TM.
\end{equation}
Similarly, we get the same bounds for the contribution from the error terms in $I_\Upsilon^\pm(x)$.

By \eqref{eqn:R-<<sum} and Lemma \ref{lemma:I-}, we have the contribution from the first term in $I_\Upsilon^\pm(x)$ is
\begin{align} \label{eqn:R-<<sum_Upsilon}
  \mathcal{R}_{\Upsilon}^{-}(N)
  &  \ll  \frac{MT}{N^{1/2}} \sum_{r\asymp \frac{\sqrt{N}}{T} }\frac{1}{r^{1/2}} \sum_{\pm }
    \bigg| \frac{N^{1/2} \Upsilon^{1/2}}{M}
    \int_{0}^{\infty}  V\left(\xi \right)\xi^{-1/2}
    \hat{k}\left(\frac{2\sqrt{N\xi}}{r M} -\frac{T}{\pi M} \right)  \\
    & \hskip 1cm \cdot
  \sum_{n\geq1}  A(1,n)  e\left(\pm \frac{n}{r}\right) n^{-1/2} V(\eta_0 + \eta_1 +\eta_2 ) e^{i h_\xi(n/r^3)} \dd \xi   \bigg| . \nonumber
\end{align}

We first consider the $n$-sum above.
Note that we have
\[
  n\asymp \frac{T^2 \Upsilon r^3}{N} \asymp \frac{N^{1/2} \Upsilon}{T},
  \quad
  \frac{n}{r} \asymp  \Upsilon ,
  \quad
  h_\xi(n/r^3) \asymp \Upsilon.
\]
This suggests to us viewing $e\left(\pm \frac{n}{r}\right)$ as an analytic weight function in  $n$ instead of an additive character. By Lemma \ref{lemma:FE3} with $w(n)=e\left(\pm \frac{n}{r}\right) (\frac{N^{1/2} \Upsilon}{T})^{1/2} n^{-1/2} V(\eta_0 + \eta_1 +\eta_2 ) e^{i h_\xi(n/r^3)}$ and  exactly the same arguments as in Lemma \ref{lemma:W1}, we have
\[
  \sum_{n\geq1}  A(1,n)  e\left(\pm \frac{n}{r}\right) n^{-1/2} V(\eta_0 + \eta_1 +\eta_2 ) e^{i h_\xi(n/r^3)}
  \ll \Upsilon^{3/2}.
\]
Hence by \eqref{eqn:xi}, we have
\[
  \mathcal{R}_{\Upsilon}^{-}(N) \ll T^\varepsilon
  \frac{MT}{N^{1/2}}  \left(\frac{\sqrt{N}}{T} \right)^{1/2}
     \frac{N^{1/2} \Upsilon^{1/2}}{M}  \frac{M}{T} \left(\frac{N^{1/2} \Upsilon}{T}\right)^{-1/2} \Upsilon^{3/2}.
\]
Recall that $\Upsilon \ll T^{1+\varepsilon}/M$ and $M\geq T^{1/3+\varepsilon}$. So we have
\begin{equation}\label{eqn:R-<<}
  \mathcal{R}_{\Upsilon}^{-}(N) \ll
  \frac{T^{3/2+\varepsilon}}{M^{1/2}} \ll T^{1+\varepsilon} M.
\end{equation}

By \eqref{eqn:S2Skl}, \eqref{eqn:Skl=D+R}, \eqref{eqn:D}, \eqref{eqn:R<<R(N)}, \eqref{eqn:R+<<}, and \eqref{eqn:R-<<}, we complete the proof of Theorem \ref{thm:moment_L_functions_6}.


\section{Mixed moment of $L$-functions}
\label{sec:mixed_moment_L_functions}

In this section, we will prove Theorem \ref{thm:moment_L_functions_2+8}.
Let $\phi_j$ be a Hecke--Maass cusp form with the spectral parameter $t_j>0$ and the Fourier coefficients $\lambda_j(n)$. The $L$-function of $\phi_j$ is defined by
\[L(s,\phi_j)= \sum_{n\geq1} \frac{\lambda_j(n)}{n^s} = \prod_{p} \prod_{\pm}\left(1-\frac{\alpha_{j}(p)^{\pm1}}{p^s}\right)^{-1}, \quad
\Re(s)>1.\]
The triple product $L$-function is given by
\[
  L(s,\phi_j\times \phi\times \phi_k) = \prod_p \prod_{\pm_1} \prod_{\pm_2}\prod_{\pm_3} \left( 1- \frac{\alpha_j(p)^{\pm_1 1} \alpha_\phi(p)^{\pm_2 1}\alpha_k(p)^{\pm_3 1}}{p^s} \right)^{-1}, \quad \Re(s)>1.
\]
These $L$-functions have analytic continuations to the whole complex plane.

By the Cauchy--Schwarz inequality and the following well known bound
\[
    \sum_{T-T^\varepsilon \leq t_j \leq T+T^\varepsilon} L(1/2,\phi_j)^2  \ll T^{1+\varepsilon},
\]
to prove Theorem \ref{thm:moment_L_functions_2+8}, it suffices to prove the following
\[
    \sum_{T-T^\varepsilon\leq t_j \leq T+T^\varepsilon}  L(1/2,\phi_j \times \phi\times \phi_k)^2 \ll T^{2+\varepsilon}.
\]

By Ramakrishnan \cite{Ramakrishnan}, we know that $\Phi=\phi\times \phi_k$ is a self-dual Hecke--Maass cusp form for $\SL(4,\mathbb{Z})$ when $\phi\neq \phi_k$. Let $A(l,m,n)$ be the normalized Fourier coefficients of $\Phi$. Then we have (\emph{cf.}\ Goldfeld \cite[\S 12.3]{Goldfeld})
\[
  L(s,\phi_j \times \Phi) = \sum_{m,n\geq1} \frac{A(1,m,n)\lambda_j(n)}{(m^2 n)^s}, \quad \Re(s)>1.
\]
Assume that $\phi_j$ is even and the root numbers of $\phi$ and $\phi_k$ are equal.
The complete $L$-function is
\[
  \Lambda(s,\phi_j \times \Phi) =L_\infty(s,\phi_j \times \Phi)L(s,\phi_j \times \Phi),
\]
where
\[
  L_\infty(s,\phi_j \times \Phi) = \pi^{-4s} \prod_{\pm_1}\prod_{\pm_2}\prod_{\pm_3}\Gamma\Big(\frac{s\pm_1 it_j \pm_2 it_k \pm_3 iT}{2}\Big).
\]
We first prove the following approximate functional equation.
\begin{lemma}\label{lemma:AFE8}
  Assume $t_j = T+O( T^\varepsilon )$ and $t_k\ll T^{\varepsilon/2} $. Then we have
  \[
    L(1/2,\phi_j \times \Phi ) \ll T^{\varepsilon} \cdot \int_{-T^{\varepsilon}}^{T^{\varepsilon}}
    \Big| \sum_{m^2n\leq T^{2+\varepsilon}}  \frac{A(1,m,n)\lambda_j(n)}{(m^2 n)^{1/2}}  (m^2n)^{-\varepsilon-it} \Big| \dd t + O(T^{-2022}).
  \]
\end{lemma}

\begin{proof}
  By Iwaniec--Kowalski \cite[Theorem 5.3]{IwaniecKowalski2004analytic}, we have
  \[
    L(1/2,\phi_j \times \Phi ) = 2 \sum_{m,n\geq1}  \frac{A(1,m,n)\lambda_j(n)}{(m^2 n)^{1/2}} V(m^2n,t_j),
  \]
  where
  \[
    V(m^2n,t_j) = \frac{1}{2\pi i} \int_{(3)} \frac{L_\infty(1/2+s,\phi_j\times \Phi)}{L_\infty(1/2,\phi_j\times \Phi)} (m^2 n)^{-s} G(s) \frac{\dd s}{s}, \quad G(s)=e^{s^2}.
  \]
  Since $t_j = T+O( T^\varepsilon )$ and $t_k\ll T^{\varepsilon/2} $, by shifting the contour to the right we know that the contribution from terms with $m^2n > T^{2+\varepsilon}$ is $O(T^{-2022})$. When $m^2n \leq T^{2+\varepsilon}$, we shift the contour to $\Re(s)=\varepsilon$, then by using the rapid decay of $G(s)$ as $|\Im(s)|> T^{\varepsilon}$ and exchange the order of summations and integral, we get
  \begin{multline*}
    L(1/2,\phi_j \times \Phi ) \ll  \int_{-T^{\varepsilon}}^{T^{\varepsilon}}
    \Big| \sum_{m^2n\leq T^{2+\varepsilon}}  \frac{A(1,m,n)\lambda_j(n)}{(m^2 n)^{1/2}}  (m^2n)^{-\varepsilon-it} \Big| \\
     \cdot \Big|\frac{L_\infty(1/2+\varepsilon+it,\phi_j\times \Phi)}{L_\infty(1/2,\phi_j\times \Phi)} \Big| e^{-t^2} \dd t + O(T^{-2022}).
  \end{multline*}
  By Stirling's formula, we complete the proof of Lemma \ref{lemma:AFE8}.
\end{proof}

By Lemmas \ref{lemma:AFE8} and \ref{lemma:large_sieve}, we have
\begin{align*}
  \sum_{T-T^\varepsilon\leq t_j \leq T+T^\varepsilon}  L(1/2,\phi_j \times \Phi)^2
  & \ll T^\varepsilon  \sum_{T-T^\varepsilon\leq t_j \leq T+T^\varepsilon}
  \int_{-T^{\varepsilon}}^{T^{\varepsilon}}
    \Big| \sum_{m^2n\leq T^{2+\varepsilon}}  \frac{A(1,m,n)\lambda_j(n)}{(m^2 n)^{1/2+\varepsilon+it}}  \Big|^2 \dd t \\
    & \hskip 8cm + O(T^{-20})  \\
  & \ll T^{2+\varepsilon}
  \sum_{n\leq T^{2+\varepsilon}} \Big|\sum_{m\leq T^{1+\varepsilon}/n^{1/2}} \frac{|A(1,m,n)|}{m  n^{1/2}}\Big|^2.
\end{align*}
By the Cauchy--Schwarz inequality, we have
\begin{align*}
  \sum_{T-T^\varepsilon\leq t_j \leq T+T^\varepsilon}  L(1/2,\phi_j \times \Phi)^2
  & \ll T^{2+\varepsilon}
  \sum_{n\leq T^{2+\varepsilon}}  \sum_{m\leq T^{1+\varepsilon}/n^{1/2}} \frac{|A(1,m,n)|^2}{m  n } .
\end{align*}
We will use the following lemma about the Ramanujan conjecture on average.

\begin{lemma}\label{lemma:Ramanujan}
  Let $\Phi$ be any Hecke--Maass cusp form for $\SL(4,\mathbb{Z})$, with normalized Fourier coefficients $A(l,m,n)$.
  Assume the spectral parameters of $\Phi$ are bounded by $T$. Then for $M,N\geq1$, we have
  \[
    \sum_{m\sim M} \sum_{n\sim N} |A(1,m,n)|^2 \ll (TMN)^\varepsilon MN.
  \]
\end{lemma}

\begin{proof}
  The proof of this is same as Chandee--Li \cite[Lemma 2.2]{CL} with additional information that the dependence on $\Phi$ for the averages of Fourier coefficients in \cite[\S3]{CL} is $O(T^\varepsilon)$ (see e.g.\ \cite{li2010upper}). Here we have used the fact  that the exterior square $\wedge^2 \Phi$ is essentially automorphic by the work of Kim \cite{Kim}.
  See \cite[Lemma 3.2]{CL}.
\end{proof}

By using dyadic decomposition of intervals and Lemma \ref{lemma:Ramanujan}, we have
\begin{align*}
  \sum_{T-T^\varepsilon\leq t_j \leq T+T^\varepsilon}  L(1/2,\phi_j \times \Phi)^2
  & \ll T^{2+\varepsilon}  .
\end{align*}
This completes the proof of Theorem \ref{thm:moment_L_functions_2+8}.


\section{Cubic moment of Hecke--Maass cusp forms} 
\label{sec:cubic_moment}


In this section, we will prove Theorems \ref{thm:HM} and \ref{thm:HM2}. We can take $\phi$ to be real valued. We use $\{\phi_j\}$ for an orthonormal basis of Hecke--Maass cusp forms  to highlight $\phi$. Assume $\phi\in \{\phi_j\}$.
Let
\[
  \mathcal{M}_k := \int_{\mathbb X} \phi_k(z) \phi(z)^3 \frac{\dd x \dd y}{y^2}
  = \langle \phi_k \phi , \phi^2 \rangle
\]
and
\[
  \mathcal{M}_t := \int_{\mathbb X} E_t(z) \phi(z)^3 \frac{\dd x \dd y}{y^2}
  = \langle  E_t \phi, \phi^2 \rangle.
\]
By Parseval's formula we get
\begin{equation}\label{eqn:Mk=Mkc+Mke}
  \mathcal{M}_k
  = \langle    \phi_k \phi,1\rangle
  \langle \frac{3}{\pi}, \phi^2  \rangle
  +
  \mathcal{M}_{k,c}
  + \mathcal{M}_{k,e}
\end{equation}
and
\begin{equation}\label{eqn:Mt=Mtc+Mte}
  \mathcal{M}_t
  = \langle E_t \phi,1 \rangle \langle  \frac{3}{\pi},\phi^2  \rangle
  +\mathcal{M}_{t,c}
  + \mathcal{M}_{t,e},
\end{equation}
where
\begin{align*}
  \mathcal{M}_{k,c} & = \sideset{}{'}\sum_{j\geq1} \langle \phi_k \phi,  \phi_j \rangle
  \langle \phi_j ,\phi^2  \rangle ,  \quad
  \mathcal{M}_{k,e}  = \frac{1}{4\pi} \int_{\mathbb{R}}
  \langle \phi_k \phi , E_\tau \rangle \langle  E_\tau, \phi^2\rangle \dd \tau, \\
  \mathcal{M}_{t,c} & = \sideset{}{'}\sum_{j\geq1} \langle E_t  \phi , \phi_j \rangle \langle \phi_j, \phi^2 \rangle,   \quad
  \mathcal{M}_{t,e}   =\frac{1}{4\pi} \int_{\mathbb{R}}
  \langle E_t  \phi, E_\tau \rangle \langle  E_\tau, \phi^2 \rangle \dd \tau.
\end{align*}
We assume that $\lambda_k \ll \lambda_\phi^\varepsilon$ and $t\ll \lambda_\phi^\varepsilon$. So we have $\langle  \phi_k \phi, 1\rangle = \langle E_t \phi,1 \rangle =0$.
Note that $\langle \phi_k \phi,  \phi_j \rangle=0$ unless the product of the root numbers is 1. We further assume that the root number of $\phi_k$ is the same as the one of $\phi$.
To estimate the contribution from the sums and integrals above we will use the Rankin--Selberg theory and Watson's formulas to reduce them to moments of $L$-functions.

\subsection{The Rankin--Selberg theory and Watson's formula}
\label{subsec:RS&Watson}

By the unfolding method in the Rankin--Selberg theory, we get (see e.g. \cite[\S7.2]{Goldfeld})
\[
   \langle \phi_k \phi , E_\tau  \rangle
   = \frac{\rho_k(1) \rho_\phi(1) \Lambda(1/2-i\tau,\phi_k\times \phi)}{2 \Lambda(1-2i\tau)},
\]
\[
  \langle  E_\tau, \phi^2\rangle
  = \frac{\rho_\phi(1)^2 \xi(1/2+i\tau) \Lambda(1/2+i\tau,\Sym^2 \phi)}{2 \Lambda(1+2i\tau)},
\]
\[
  \langle  E_t \phi,\phi_j\rangle
  = \frac{\rho_j(1) \rho_\phi(1) \Lambda(1/2+it,\phi_j\times \phi)}{2 \Lambda(1+2it)},
\]
\[
  \langle E_t \phi, E_\tau \rangle
  = \frac{\rho_t(1) \rho_\phi(1) \Lambda(1/2-i\tau+it,  \phi)\Lambda(1/2-i\tau-it, \phi)}{2 \Lambda(1-2i\tau)} .
\]

By Watson's formula \cite{watson2008rankin}, we have
\[
  |\langle \phi_k \phi,  \phi_j \rangle |^2
  = \frac{\Lambda(1/2,\phi_k\times \phi\times \phi_j)}{8 \Lambda(1,\Sym^2 \phi_k)\Lambda(1,\Sym^2 \phi)\Lambda(1,\Sym^2 \phi_j)}
\]
and
\[
  |\langle \phi_j ,\phi^2  \rangle|^2
  = \frac{\Lambda(1/2,\phi_j) \Lambda(1/2, \Sym^2 \phi\times \phi_j)}{8  \Lambda(1,\Sym^2 \phi)^2 \Lambda(1,\Sym^2 \phi_j)}.
\]

\subsection{Estimate of $\mathcal{M}_{k,c}$}
\label{subsec:Mkc}
By Watson's formulas above, we have
\[
  |\mathcal{M}_{k,c}|  \leq  \sum_{j\geq1}  \frac{\Lambda(1/2,\phi_k\times \phi\times \phi_j)^{1/2} \Lambda(1/2,\phi_j)^{1/2} \Lambda(1/2, \Sym^2 \phi\times \phi_j)^{1/2} } {8 \Lambda(1,\Sym^2 \phi_k)^{1/2} \Lambda(1,\Sym^2 \phi_j)  \Lambda(1,\Sym^2 \phi)^{3/2}  }.
\]
By the definition of complete $L$-functions, we have
\[
  \mathcal{M}_{k,c}  \ll  \sum_{j\geq1}  \frac{L(1/2,\phi_k\times \phi\times \phi_j)^{1/2} L(1/2,\phi_j)^{1/2} L(1/2, \Sym^2 \phi\times \phi_j)^{1/2} } {L(1,\Sym^2 \phi_k)^{1/2} L(1,\Sym^2 \phi_j)  L(1,\Sym^2 \phi)^{3/2}  } H(t_j;T,t_k),
\]
where
\[
  H(t_j;T,t_k) = \frac{L_\infty(1/2,\phi_k\times \phi\times \phi_j)^{1/2} L_\infty(1/2,\phi_j)^{1/2} L_\infty(1/2, \Sym^2 \phi\times \phi_j)^{1/2} } {L_\infty(1,\Sym^2 \phi_k)^{1/2} L_\infty(1,\Sym^2 \phi_j)  L_\infty(1,\Sym^2 \phi)^{3/2}  }.
\]
By Stirling's formula we have
\begin{multline*}
  H(t_j;T,t_k) \ll \exp\Big(-\frac{\pi}{2} Q(t_j;T,t_k)\Big) (t_j+T+t_k)^{-1/4} (t_j+T-t_k)^{-1/4} \\
  \cdot (1+|t_j-T+t_k|)^{-1/4}(1+|t_j-T-t_k|)^{-1/4} t_j^{-1/2} (t_j+2T)^{-1/4} (1+|t_j-2T|)^{-1/4},
\end{multline*}
with
\begin{multline}\label{eqn:Q}
    Q(t_j;T,t_k) = \frac{|t_j+T+t_k|}{2}+\frac{|t_j+T-t_k|}{2}
    + \frac{|t_j-T+t_k|}{2} + \frac{|t_j-T-t_k|}{2} \\
    + t_j + \frac{|t_j+2T|}{2} + \frac{|t_j-2T|}{2} - t_k - 2 t_j -3T.
\end{multline}
Note that for $t_k <  T$, we have
\begin{equation*}
  Q(t_j;T,t_k) = \left\{
  \begin{array}{ll}
    2t_j -3T-t_k, & \textrm{if $t_j\geq 2T$}, \\
    t_j-T-t_k, &  \textrm{if $T+t_k \leq t_j< 2T$} , \\
    0, &  \textrm{if $T-t_k \leq t_j< T+t_k$} , \\
    T-t_j-t_k, &  \textrm{if $ t_j< T-t_k$} .
  \end{array}\right.
\end{equation*}
Hence for $t_k \leq T^{\varepsilon/2}$, we have
\begin{equation*}\label{eqn:H<<1}
  H(t_j;T,t_k) \ll t_j^{-2022} T^{-2022},
\end{equation*}
if $t_j\geq T+T^{\varepsilon}$ or $t_j \leq T- T^\varepsilon$; and we have
\begin{equation}\label{eqn:H<<}
  H(t_j;T,t_k) \ll   T^{-3/2},
\end{equation}
if $|t_j-T|\leq T^\varepsilon$. So by the convexity bounds for $L$-functions we have
\begin{multline*}
  \mathcal{M}_{k,c}  \ll \frac{1}{T^{3/2}} \sum_{|t_j-T|\leq T^\varepsilon}  L(1/2,\phi_k\times \phi\times \phi_j)^{1/2} L(1/2,\phi_j)^{1/2} L(1/2, \Sym^2 \phi\times \phi_j)^{1/2}
  + O(T^{-20}).
\end{multline*}
By the Cauchy--Schwarz inequality, we have
\begin{multline*}\label{eqn:Mkc<<product}
  \mathcal{M}_{k,c}  \ll \frac{1}{T^{3/2}} \Big(\sum_{|t_j-T|\leq T^\varepsilon}  L(1/2,\phi_k\times \phi\times \phi_j)  L(1/2,\phi_j) \Big)^{1/2} \\
  \cdot
  \Big(\sum_{|t_j-T|\leq T^\varepsilon}    L(1/2, \Sym^2 \phi\times \phi_j) \Big)^{1/2}
  + O(T^{-20}).
\end{multline*}
By Theorems \ref{thm:moment_L_functions_6} and \ref{thm:moment_L_functions_2+8}, we get
\begin{equation}\label{eqn:Mkc<<}
   \mathcal{M}_{k,c} \ll T^{-3/2+3/4+\varepsilon+2/3} \ll T^{-1/12+\varepsilon}.
\end{equation}

\subsection{Estimate of $\mathcal{M}_{k,e}$}
\label{subsec:Mke}
By the formulas in \S \ref{subsec:RS&Watson}, we have
\begin{align*}
  \mathcal{M}_{k,e} & \ll  \int_{\mathbb{R}}
  |\langle \phi_k \phi , E_\tau \rangle \langle  E_\tau, \phi^2\rangle| \dd \tau \\
  & \ll \int_{\mathbb{R}} \frac{|\rho_k(1)| |\rho_\phi(1)|^3 |\Lambda(1/2-i\tau,\phi_k\times \phi) \Lambda(1/2+i\tau) \Lambda(1/2+i\tau,\Sym^2 \phi)| }{ |\xi(1-2i\tau)|^2 }
    \dd \tau.
\end{align*}
Hence \eqref{eq:rho1squared} and \eqref{eq:rho1squared-c}, we have
\begin{equation*}
  \mathcal{M}_{k,e}
  \ll \int_{\mathbb{R}} \frac{ |L(1/2-i\tau,\phi_k\times \phi) \zeta(1/2+i\tau) L(1/2+i\tau,\Sym^2 \phi)|} { L(1,\Sym^2 \phi_k)^{1/2} L(1,\Sym^2 \phi)^{3/2} |\zeta(1-2i\tau)|^2 } H_{k,e}(\tau;T,t_k)
    \dd \tau,
\end{equation*}
where
\begin{align*}
  H_{k,e}(\tau;T,t_k) & = \exp\Big(\frac{\pi (t_k+3T+2|\tau|)}{2}\Big) \\
  & \hskip 1cm \cdot
   |L_\infty(1/2-i\tau,\phi_k\times \phi) L_\infty(1/2+i\tau) L_\infty(1/2+i\tau,\Sym^2 \phi)| \\
  & \ll
  \exp\Big( -\frac{\pi}{2} Q(|\tau|;T,t_k) \Big)
  (1+|\tau+T+t_k|)^{-1/4} (1+|\tau+T-t_k|)^{-1/4} \\
  & \hskip 1cm \cdot (1+|\tau-T+t_k|)^{-1/4}(1+|\tau-T-t_k|)^{-1/4}
  (1+|\tau|)^{-1/2} \\
  & \hskip 1cm \cdot (1+|\tau+2T|)^{-1/4}(1+|\tau-2T|)^{-1/4},
\end{align*}
with $Q$ as in \eqref{eqn:Q}.
 Hence  we have
\begin{equation*}
  \mathcal{M}_{k,e}
  \ll  T^{-3/2+\varepsilon}  \int_{|\tau|-T \ll T^\varepsilon}
   |L(1/2-i\tau,\phi_k\times \phi) \zeta(1/2+i\tau) L(1/2+i\tau,\Sym^2 \phi)|  \dd \tau
   + O(T^{-20}).
\end{equation*}
Note that for $|\tau|-T \ll T^\varepsilon$ and $t_k\ll T^{\varepsilon/2}$, we at least have the following convexity bound
\begin{equation}\label{eqn:convexity}
  L(1/2-i\tau,\phi_k\times \phi) \ll T^{1/2+\varepsilon},
\end{equation}
 the Weyl bound
\begin{equation}\label{eqn:subconvexity1}
  \zeta(1/2+i\tau) \ll T^{1/6+\varepsilon} ,
\end{equation}
and the subconvexity bound (\cite{KY} or Corollary \ref{cor})
\begin{equation}\label{eqn:subconvexity3}
  L(1/2+i\tau,\Sym^2 \phi) \ll T^{2/3+\varepsilon}.
\end{equation}
Hence  we obtain
\begin{equation}\label{eqn:Mke<<}
  \mathcal{M}_{k,e}
  \ll  T^{-3/2+1/2+1/6+2/3+\varepsilon}  \ll  T^{-1/6+\varepsilon}.
\end{equation}

\subsection{Estimate of $\mathcal{M}_{t,c}$}
\label{subsec:Mtc}
By the formulas in \S \ref{subsec:RS&Watson} and similar estimates as in \S \ref{subsec:Mkc}, we have
\begin{multline*}
  \mathcal{M}_{t,c}  \ll \frac{1}{T^{3/2}} \sum_{|t_j-T|\leq T^\varepsilon}
  L(1/2,E_t\times \phi\times \phi_j)^{1/2}  L(1/2,\phi_j)^{1/2} L(1/2, \Sym^2 \phi\times \phi_j)^{1/2}
  + O(T^{-20}).
\end{multline*}
Note that
\[
  L(1/2,E_t\times \phi\times \phi_j)^{1/2} = |L(1/2+it, \phi\times \phi_j)|
\]
and we have the subconvexity bound (\cite{Ivic})
\begin{equation}\label{eqn:subconvexity2}
  L(1/2,\phi_j) \ll t_j^{1/3+\varepsilon}.
\end{equation}
Hence we get
\begin{equation*}
  \mathcal{M}_{t,c}  \ll  T^{-4/3+\varepsilon}  \sum_{|t_j-T|\leq T^\varepsilon}  |L(1/2+it, \phi\times \phi_j)|   L(1/2, \Sym^2 \phi\times \phi_j)^{1/2}
  + O(T^{-20}).
\end{equation*}
By the Cauchy--Schwarz inequality, we have
\begin{multline*}
  \mathcal{M}_{t,c}  \ll  T^{-4/3+\varepsilon}  \Big( \sum_{|t_j-T|\leq T^\varepsilon}  |L(1/2+it, \phi\times \phi_j)|^2  \Big)^{1/2}
  \Big( \sum_{|t_j-T|\leq T^\varepsilon}     L(1/2, \Sym^2 \phi\times \phi_j) \Big)^{1/2} \\
  + O(T^{-20}) .
\end{multline*}
Note that by the spectral large sieve inequality, we have  (\emph{cf.} Theorem \ref{thm:moment_L_functions_2+8})
\[
   \sum_{|t_j-T|\leq T^\varepsilon}  |L(1/2+it, \phi\times \phi_j)|^2
   \ll T^{1+\varepsilon}.
\]
Together with Theorem \ref{thm:moment_L_functions_6}, we have
\begin{equation}\label{eqn:Mtc<<}
  \mathcal{M}_{t,c}  \ll  T^{-4/3+ 1/2+2/3+\varepsilon} = T^{-1/6+\varepsilon}.
\end{equation}

\subsection{Estimate of $\mathcal{M}_{t,e}$}
\label{subsec:Mte}
By the formulas in \S \ref{subsec:RS&Watson} and similar estimates as in \S \ref{subsec:Mke}, we have
\begin{equation*}
  \mathcal{M}_{t,e}
  \ll  T^{-3/2+\varepsilon}  \int_{|\tau|-T \ll T^\varepsilon}
   |L(1/2-i\tau,E_t\times \phi) \zeta(1/2+i\tau) L(1/2+i\tau,\Sym^2 \phi)|  \dd \tau
   + O(T^{-20}).
\end{equation*}
Note that for $|\tau|-T \ll T^\varepsilon$, we have the following convexity bounds
\[
  |L(1/2-i\tau,E_t\times \phi)| = |L(1/2-i\tau+it, \phi) L(1/2-i\tau-it,\phi)|
  \ll T^{1/2+\varepsilon}.
\]
Together with \eqref{eqn:convexity} and \eqref{eqn:subconvexity3}, we get
\begin{equation}\label{eqn:Mte<<}
  \mathcal{M}_{t,e}
  \ll  T^{-3/2+1/2+1/6+2/3+\varepsilon}  = T^{-1/6+\varepsilon}.
\end{equation}

\subsection{Proof of Theorems \ref{thm:HM} and \ref{thm:HM2}}

By \eqref{eqn:Mk=Mkc+Mke}, \eqref{eqn:Mkc<<} and \eqref{eqn:Mke<<}, we have
\[
  \mathcal{M}_k \ll T^{-1/12+\varepsilon},
\]
which proves \eqref{eqn:3moment_cusp}.
By \eqref{eqn:Mt=Mtc+Mte}, \eqref{eqn:Mtc<<} and \eqref{eqn:Mte<<}, we have
\[
  \mathcal{M}_t \ll T^{-1/6+\varepsilon},
\]
which proves \eqref{eqn:3moment_ES}.
This completes the proof of Theorem \ref{thm:HM2}.

Note that for any $l\geq1$
\[
  \langle \phi_k, \psi \rangle = \frac{1}{(1/4+t_k^2)^l}\langle \Delta^l \phi_k, \psi \rangle  = \frac{1}{(1/4+t_k^2)^l}\langle \phi_k, \Delta^l \psi \rangle \ll_l \frac{1}{(1/4+t_k^2)^l}  ,
\]
\[
  \langle E_t, \psi \rangle  = \frac{1}{(1/4+t^2)^l}\langle E_t, \Delta^l \psi \rangle
  \ll \frac{1}{(1/4+t^2)^l} \int_{\mathbb X}  |E_t(z)| |\Delta^l \psi(z)| \frac{\dd x \dd y}{y^2}.
\]
By the Cauchy--Schwarz inequality and QUE for Eisenstein series \eqref{eqn:QUE_ES}, we have
\begin{equation}\label{eqn:<phik,psi><<}
  \langle \phi_k, \psi \rangle
  \ll_A t_k^{-A}, \quad
  \langle E_t, \psi \rangle
  \ll_A (1+|t|)^{-A}, \quad
  {\rm for \ any} \ A>0.
\end{equation}
By using Parseval's formula and Watson's formula  as in the proof of Theorem \ref{thm:HM2}, together with the convexity bounds of central $L$-values, we have
\begin{equation}\label{eqn:<phik,phi3><<}
  \langle \phi_k ,\phi^3\rangle \ll (t_k T)^B
  \quad \textrm{and} \quad
  \langle E_t ,\phi^3\rangle \ll ((1+|t|) T)^B,
\end{equation}
for some absolute constant $B>0$.
By \eqref{eqn:Selberg_decomposition}, \eqref{eqn:<phik,psi><<}, and \eqref{eqn:<phik,phi3><<}, the contribution to $\langle \psi,\phi^3\rangle$ from $\phi_k$ and $E_t$ with $t_k\geq T^\varepsilon$ and $|t|\geq T^\varepsilon$ is negligibly small.
Hence we prove Theorem \ref{thm:HM} by using Theorem \ref{thm:HM2}.

\section*{Acknowledgements}

The author would like to thank Prof.\ Jianya Liu and Ze\'{e}v Rudnick for their valuable discussions and constant encouragement.
He also wants to thank  Peter Humphries and Yongxiao Lin for their comments.
He gratefully thanks to the referees for the constructive comments and recommendations which definitely improve the readability and quality of the paper.

\textbf{Conflict of interest.} The  author states that there is no conflict of interest.

\textbf{Data Availability.} Data sharing not applicable to this article as no datasets were generated or analysed during the current study.


\end{document}